\documentclass[11pt, a4paper, oneside,reqno]{amsart}

\makeatletter
\def\@settitle{\begin{center}%
		\baselineskip14\p@\relax
		\normalfont\LARGE\bfseries
		\@title
	\end{center}%
}

\def\section{\@startsection{section}{1}%
	\z@{.7\linespacing\@plus\linespacing}{.5\linespacing}%
	{\normalfont\large\bfseries}}

\def\subsection{\@startsection{subsection}{2}%
	\z@{.5\linespacing\@plus.7\linespacing}{.5\linespacing}%
	{\normalfont\bfseries}}

\def\@setauthors{%
  \begingroup
  \def\thanks{\protect\thanks@warning}%
  \trivlist
  \centering\footnotesize \@topsep30\p@\relax
  \advance\@topsep by -\baselineskip
  \item\relax
  \author@andify\authors
  \def\\{\protect\linebreak}%
  \authors%
  \ifx\@empty\contribs
  \else
    ,\penalty-3 \space \@setcontribs
    \@closetoccontribs
  \fi
  \endtrivlist
  \endgroup
}

\makeatother

\usepackage[table, xcdraw, dvipsnames]{xcolor}
\definecolor{darkblue}{rgb}{0.0, 0.0, 0.45}
\definecolor{darkgreen}{rgb}{0.0, 0.45, 0}
\usepackage[colorlinks	= true,
raiselinks	= true,
linkcolor	= darkblue, 
citecolor	= Mahogany,
urlcolor	= darkgreen,
pdfauthor	= {},
pdftitle	= {},
pdfkeywords	= {},
pdfsubject	= {},
plainpages	= false]{hyperref}

\allowdisplaybreaks
\pdfoutput=1
\date{\today}

\usepackage{dsfont,amsfonts,amssymb,amsmath,amsthm}
\usepackage{mathtools, mathrsfs}

\usepackage{graphicx}
\usepackage[font=small,labelfont=rm]{subcaption}
\usepackage[font=small,margin=10pt]{caption}
\usepackage{wrapfig}

\usepackage{multirow}
\usepackage{booktabs}

\usepackage{enumitem}
\usepackage{tikz}
\usepackage{courier} 

\usepackage{nicefrac}

\theoremstyle{plain}
\newtheorem{Thm}{Theorem}[section]

\newtheorem{Prop}[Thm]{Proposition}

\newtheorem{Lem}[Thm]{Lemma}
\newtheorem{Cor}[Thm]{Corollary}
\newtheorem{As}[Thm]{Assumption}

\newtheorem{Rem}[Thm]{Remark}

\newtheorem{Ex}[Thm]{Example}

\DeclareMathOperator*{\argmin}{\arg\!\min}
\DeclareMathOperator*{\argmax}{\arg\!\max}

\newcommand{\R}{\mathbf{R}}

\newcommand{\N}{\mathbf{N}}

\newcommand{\ra}{\rightarrow}

\newcommand{\ol}[1]{\overline{#1}}

\newcommand{\Let}{\coloneqq}

\newcommand{\tr}{^{\top}}

\def\ssum{\begingroup\textstyle \sum\endgroup}

\newcommand{\opt}{_\star}

\newcommand{\EE}{\mathds{E}}

\usepackage[square,numbers]{natbib}
\usepackage{bm}

\usepackage{algorithm, algpseudocode}
\algrenewcommand\algorithmicrequire{\textbf{Input:}}
\algrenewcommand\algorithmicensure{\textbf{Output:}}

\usepackage{tikz}
\usetikzlibrary{arrows.meta, calc, positioning, decorations.pathmorphing, patterns}
\usetikzlibrary{decorations.pathreplacing}
\usetikzlibrary{calc,arrows.meta}

\usepackage{xargs}
\usepackage{booktabs}
\usepackage{diagbox}

\usepackage{scalerel}

\newcommand{\Trc}{\operatorname{tr}}
\newcommand{\range}{\operatorname{range}}
\newcommand{\Acl}{A}
\newcommand{\Ccl}{C}
\newcommand{\Bu}{B}
\newcommand{\Bw}{E}
\newcommand{\Pw}{\mathsf{P}}
\newcommand{\Qw}{\mathsf{Q}}
\newcommand{\Sp}{\Sigma_{\Pw}}

\newcommand{\Ks}{\mathcal{K}_{\mathrm{s}}}
\newcommand{\Wo}{X}
\newcommand{\Wc}{Y}
\newcommand{\Gzw}{G}

\newcommand{\F}{\scaleto{\mathrm{F}}{4pt}}
\newcommand{\St}{\scaleto{\mathrm{S}}{4pt}}
\newcommand{\DR}{\scaleto{\mathrm{D}}{4pt}}
\newcommand{\KL}{\scaleto{\mathrm{KL}}{4pt}}
\newcommand{\WA}{\scaleto{\mathrm{W}}{4pt}}
\newcommand{\D}{\mathcal{D}}
\newcommand{\PS}{\mathcal{P}}
\newcommand{\LO}{\mathbb{L}}
\newcommand{\LC}{\mathbb{L}^*}

\title[]{Benign Geometry and Distributional Robustness\\ of $\mathcal{H}_2$ Synthesis}

\author[]{Arman~Sharifi~Kolarijani$^{1}$, Peyman Mohajerin Esfahani$^{2}$, Tam\'{a}s Keviczky$^3$,\\
and Mohamad Amin Sharifi Kolarijani$^{3}$\\
\\
$^1$Alpha Brain Technologies, The Netherlands \\
$^2$University of Toronto, Canada \\
$^3$Delft University of Technology, The Netherlands}
\thanks{Correspondence to: Mohamad Amin Sharifi Kolarijani $<$\texttt{m.a.sharifikolarijani@tudelft.nl}$>$.}
\thanks{This work was partially supported by the Horizon Europe Pathfinder Open project RELIEVE-101099481.} 

\begin{document}

\begin{abstract}
In this paper, we study standard and distributionally robust $\mathcal{H}_2$ synthesis problem of a stabilizing state-feedback controller for discrete-time linear time-invariant systems. 
Without requiring a nonsingular disturbance controllability Gramian
or a positive-definite control penalty, we establish that stationarity of a stabilizing gain in standard $\mathcal{H}_2$ synthesis is equivalent to global optimality and derive an exact Pythagorean identity for the performance difference.
We then generalize the $\mathcal{H}_2$ synthesis problem by considering i.i.d.~disturbances with zero mean and uniformly bounded second moments, and extend the stationarity-optimality equivalence to the corresponding distributionally robust $\mathcal{H}_2$ synthesis.
Moreover, we show that every standard $\mathcal{H}_2$-optimal gain is simultaneously optimal for every such distributionally robust problem.
Finally, we specialize the framework to Frobenius, Kullback--Leibler, and Wasserstein-2 ambiguity sets and obtain explicit characterizations of their worst-case covariances, which, in turn, provide performance certificates for a common optimal controller.


\end{abstract}

\maketitle

\section{Introduction}
\label{sec:intro}

Reliable feedback design in the presence of uncertainty and exogenous disturbances is a fundamental challenge across control applications.
Robust control provides a systematic framework for addressing this challenge, with influence extending beyond classical engineering to economics~\cite{hansen2008robustness}, networked systems~\cite{cohen2010complex}, biological systems~\cite{kitano2004biological}, and optimization~\cite{bental2009robust}.

Among the principal robust-performance frameworks are $\mathcal{H}_2$, $\mathcal{H}_\infty$, and mixed $\mathcal{H}_2/\mathcal{H}_\infty$ synthesis~\cite{Doyle1989,ZhouDoyleGlover1996}.
The $\mathcal{H}_2$ criterion in particular minimizes the closed-loop impulse-response energy, or, equivalently, the steady-state performance output power under zero-mean white noise with unit covariance. 
The main difficulty is of course that the $\mathcal{H}_2$ synthesis leads to a nonconvex problem, even in the case of static state-feedback design. 
However, the state-feedback $\mathcal{H}_2$ synthesis problem admits a convex linear matrix inequality (LMI) formulation using Lyapunov inequalities and trace bounds on closed-loop
Gramians~\cite{boyd1994lmi,CaverlyForbes2019LMI}, after suitable changes of variables~\cite{meilakhs1975stabilization}. 
Another classical synthesis method uses the algebraic Riccati equation (ARE)~\cite{ZhouDoyleGlover1996}.
Indeed, state-feedback $\mathcal{H}_2$ synthesis is closely related to linear quadratic regulation (LQR) problem by interpreting the squared performance output as a quadratic state-and-control cost~\cite{Kalman1960Contributions,trentelman1995sampled}. 
Under the usual regular LQR assumptions, the stabilizing Riccati solution yields an $\mathcal{H}_2$-optimal gain.

The connection with LQR motivates studying optimality directly in the feedback gain. For regular LQR, a primal-dual formulation yields a first-order condition with respect to the gain~\cite{LeeHu2019}. 
Gradient-dominance bounds establish that first-order stationarity of the LQR cost is \emph{sufficient} for global optimality~\cite{FazelGeKakadeMesbahi2018,Bu2019LQRFirstOrder,watanabe2026gradient}. 
Such bounds also underpin global convergence guarantees for suitable policy-gradient methods.
However, these results require certain assumptions on the problem data that, among others, include positive-definite control penalty and positive definiteness of the accumulated state covariance~\cite[Sec.~3.3]{watanabe2026gradient}. 
The latter condition, in particular, translates to nonsingularity of the disturbance controllability Gramian in the context of $\mathcal{H}_2$ synthesis. 
This motivates examining stationarity and global optimality in $\mathcal{H}_2$ synthesis without requiring these conditions. 

The standard formulation of $\mathcal{H}_2$ synthesis prescribes the disturbance covariance.
However, in many applications, the underlying distribution is often estimated from finite data, affected by distribution shift, or only partially known.
Distributionally robust optimization (DRO)~\cite{DelageYe2010,Wiesemann2014,RahimianMehrotra2019} addresses such misspecification by optimizing against the worst-case probability distribution in an ambiguity set.
In its generic form, DRO seeks a decision $x$ solving $\inf_{x \in \mathcal{X}} \sup_{\Pw \in \D} \EE_{\xi\sim\Pw}[\ell(x,\xi)]$,
where $\ell$ is a loss function and $\D$ encodes the available distributional information.
Different ambiguity sets induce different notions of robustness.
Moment-based sets constrain low-order statistics and are attractive when only means and covariances can be estimated reliably~\cite{DelageYe2010,Wiesemann2014}.
Kullback--Leibler (KL) divergence produces relative-entropy neighborhoods and exponential-tilting representations~\cite{hu2013kullback}.
Wasserstein ambiguity sets compare distributions through optimal transport, accommodate support mismatch, and often yield finite-sample guarantees and tractable reformulations in data-driven settings~\cite{EsfahaniKuhn2018,GaoKleywegt2023}.
In control, moment-based ambiguity has been used to enforce distributionally robust chance and conditional value-at-risk constraints~\cite{vanparys2016distributionally} and to address LQR problems with multiplicative disturbance~\cite{Coppens2020DRLQR}.
Another line of research formulates minimax state-feedback control with adversarial conditional disturbance distributions, using Wasserstein ambiguity or penalties~\cite{yang2021wasserstein,kim2023distributional}. 
Stage-wise relative-entropy bounds on the conditional disturbance distributions are also used in related linear quadratic Gaussian (LQG) formulations~\cite{Falconi2025DistributedUncertainty}. 
For a class of minimax problems with trajectory-level entropy constraints, relative-entropy duality yields a risk-sensitive reformulation~\cite{petersen2000minimax} with exponential criteria of the kind used in classical risk-sensitive linear control~\cite{Jacobson1973Exponential,Whittle1981RiskSensitive}. 
For partially observed systems controlled through output feedback, optimality of linear policies has been established under Wasserstein and KL ambiguity on temporally independent disturbances~\cite{Taskesen2023DRLQ,Fochesato2025KLLQG}.
Related studies impose a common disturbance distribution across time~\cite{Schoebi2026StationaryDRLQG} or consider infinite-horizon constrained control with temporal dependence~\cite{Brouillon2025DRInfiniteHorizon}.
More directly, in~\cite{Gramlich2025DRLMI}, the authors connect Wasserstein distributionally robust control to $\mathcal{H}_2$ synthesis and derive convex LMI formulations for independent and temporally correlated disturbances.
These formulations differ in their information patterns, temporal assumptions, and admissible disturbance distributions.
For state-feedback with zero-mean i.i.d.\ additive disturbances, the regular LQR gain is already independent of the disturbance covariance and consequently optimal across the corresponding distributional ambiguity sets, as also noted in~\cite{Taskesen2023DRLQ}. 
We note, however, that uncertainty in disturbance statistics changes the performance guarantee associated with a controller, even when the same gain remains optimal.
These observations motivate examining the persistence of nominal optimality under distributional ambiguity beyond the regular LQR setting.

In this paper, we consider a DR framework for state-feedback $\mathcal{H}_2$ synthesis for discrete-time linear time-invariant (LTI) systems with full state observation. 
The ambiguity concerns the common marginal distribution of an i.i.d.\ zero-mean disturbance sequence that belongs to a set $\D$ that is independent of the gain and has uniformly bounded second moments. 
Following the direct gain-optimization viewpoint originating with~\cite{Kalman1960Contributions} and extended to output feedback in~\cite{Levine1970ConstantOFB}, we consider $\inf_{K \in \Ks} \sup_{\Pw \in \D} v(K, \Pw)$, 
where $\Ks$ is the set of stabilizing gains and $v(K,\Pw)$ is the generalized $\mathcal{H}_2$ cost accounting for a generic disturbance distribution~$\Pw$. 
In particular, $v(K, \Pw)$ reduces to the standard $\mathcal{H}_2$ cost for $\Pw = \mathcal{N}(0,I)$, that is, normal distribution with zero mean and unit covariance. 
Our analysis only assumes attainment of the optimal value in the standard $\mathcal{H}_2$ synthesis problem and does not impose any extra restrictions on the problem data  such as nonsingularity of disturbance controllability Gramian or positive-definiteness of control penalty.
Our \textbf{main contributions} are:
\begin{itemize}
    \item \textbf{Sufficiency of stationarity for $\mathcal{H}_2$-optimality:} 
    We show that first-order stationarity is necessary and sufficient for global optimality for the DR $\mathcal{H}_2$ synthesis (Theorem~\ref{thm:DR_H2_optimal_grad}). 
    For the standard $\mathcal{H}_2$ synthesis this result follows from  an exact Pythagorean performance-difference identity (Theorem~\ref{thm:standard_H2_optimal_grad}). 

    \item \textbf{Distributionally robustness of standard $\mathcal{H}_2$-optimal gain:} 
    Exploiting the Loewner minimality of their \emph{compressed} observability Gramian (Corollary~\ref{cor:obs_gram_minimal}), we show that \emph{every} standard $\mathcal{H}_2$-optimal gain is simultaneously optimal for the considered DR $\mathcal{H}_2$ synthesis problems (Theorem~\ref{thm:dr_H2_standard}). 

    \item \textbf{Performance certificates for specific ambiguity sets:} 
    We obtain explicit worst-case covariances/distributions and provide robustness certifications for the DR $\mathcal{H}_2$ synthesis problems with ambiguity sets defined based on the Frobenius, KL, and Wasserstein-2 (W2) balls around a nominal covariance/distribution (Propositions~\ref{prop: DR cost frob}, \ref{prop: DR cost KL}, and \ref{prop: DR cost W2}).
\end{itemize}

The remainder of the paper is organized as follows. 
In Section~\ref{sec:standard_H2}, we revisit standard $\mathcal{H}_2$ synthesis and analyze the regularity properties of the corresponding cost and optimal solution. 
In Section~\ref{sec:DR-H2}, we generalize $\mathcal{H}_2$ cost by considering i.i.d., zero-mean disturbances with arbitrary covariance, and analyze the regularity properties of the corresponding DR cost and optimal solution. 
In Section~\ref{sec:specific-ambiguity-set}, we specialize the general results of Section~\ref{sec:DR-H2} to three classes of ambiguity sets. 
Section~\ref{sec:conclusion} concludes the paper with some final remarks on the limitations of the present paper.

\noindent\textbf{Notations.} 
The set of real numbers is denoted by $\R$. We also denote $\N_0 \Let \{0,1,\ldots\}$. 
For a vector $x\in\R^n$, the 2-norm is denoted by $\|x\|_2 = \sqrt{x\tr x}$. 
For a matrix $M \in \R^{n\times n}$, $\Trc(M)$ is the trace of $M$ (i.e., the sum of diagonal entries of $M$), and $\rho(M)$ is spectral radius of $M$ (i.e., the largest absolute value among its eigenvalues). 
For a symmetric matrix $M \in \R^{n\times n}$, we use $M\succeq 0$ (respectively, $M\succ 0$) to denote $M$ is positive semi-definite (respectively, positive definite). 
We use $I$ to denote the identity matrix. 
For two matrices $M,N\in\R^{n\times m}$, $M\tr$ is the transpose of $M$, $\langle M,N\rangle = \Trc(M\tr N)$ is the (Frobenius) inner product of $M$ and $N$, $\|M\|_{F} = \sqrt{\langle M,N\rangle}$ is the Frobenius norm of $M$, and $\range(M)$ is the range space of $M$. 
$\mathcal{B}(\R^p)$ is the set of Borel probability measures on $\mathbb{R}^p$. 
We use $w\sim \Pw \in\mathcal{B}(\R^p)$ to denote that the random vector $w\in\R^p$ has the distribution $\Pw$.
$\EE[\cdot]$ denotes the expectation operation; we occasionally use the subscript as in $\EE_{\Pw}[\cdot]$ to emphasize that the expectation is with respect to a random variable with distribution~$\Pw$. 
For a random vector $w\sim\Pw \in \mathcal{B}(\R^p)$, we use $\mu_{\Pw} \Let \EE_{\Pw}[w]$ and
$\Sp \Let \EE_{\Pw}[w w\tr]$ to denote its first and second moments, respectively. 
For distributions with zero mean then $\Sp$ denotes also the covariance.
We use $\mathcal{N}(\mu,\Sigma)$ to denote the Gaussian distribution with mean $\mu$ and covariance $\Sigma$. 
For a linear subspace $\mathcal{L} \subset \R^{m\times n}$, we use $\mathcal{L}^{\perp}$ to denote the orthogonal complement of $\mathcal{L}$ with respect to the Frobenius inner product. 
Throughout the paper, we use $K$ to denote the static state-feedback gain. 
To avoid clutter, we exclusively use $K$ in the subscript to show the dependence on $K$; e.g., $\Acl_K$ denotes the dependence of the variable~$\Acl_K$ on the gain~$K$.

\section{Standard $\mathcal{H}_2$ synthesis}
\label{sec:standard_H2}

In this section, we revisit the standard full-state $\mathcal{H}_2$ optimal control problem. 
We analyze the regularity properties of the corresponding cost, and, 
in particular, show that first-order stationarity is a \emph{sufficient} condition for optimality of a stabilizing gain. 

\subsection{Standard $\mathcal{H}_2$ cost}

Consider the discrete-time LTI system
\begin{equation}\label{eq:standard LTI}
\left\{\begin{array}{rcl}
     x_{t+1} &=& A x_t + \Bu u_t + \Bw w_{t}, \\
     z_t &=& Cx_t+Du_t,
\end{array}
\right.\quad t\in\N_0, 
\end{equation}
where $x_t \in \mathbb{R}^n$ is the system state, $u_t \in \mathbb{R}^m$ is the control input, $w_t \in \mathbb{R}^p$ is a disturbance input, and $z_t\in \R^q$ is the performance signal. 
$A$, $\Bu$, $\Bw$, $C$, and $D$ are matrices of compatible dimensions that describe the dynamics of the system. 
Under linear state-feedback law $u_t = K x_t$, where $K \in \R^{m \times n}$ is the gain, the closed-loop system becomes
\begin{equation}\label{eq:closed-loop}
\left\{\begin{array}{rcl}
     x_{t+1} &=& \Acl_K x_t + \Bw w_{t},\\
     z_t &=& \Ccl_K x_t,
\end{array}\right.\quad \text{where} \quad 
\left\{\begin{array}{l}
     \Acl_K \Let A + \Bu K,\\
     \Ccl_K \Let C + D K.
\end{array}\right.
\end{equation}
In this paper, we are concerned with the $\mathcal{H}_2$ synthesis of a stabilizing state-feedback controller. 
Let us define 
\begin{equation*}
    \Ks:=\{K \in\R^{m\times n}\;:\;\rho(\Acl_K)<1\},
\end{equation*}
to be the set of all stabilizing controllers for which $\Acl_K$ is Schur. 
Notice that $\Ks$ non-empty if and only if the pair $(A,\Bu)$ is stabilizable. 
The standard $\mathcal{H}_2$ synthesis is then formulated as 
\begin{equation}\label{eq:standard_H2_problem}
    \inf_{K \in \Ks} \left\{ v^{\St}(K) \Let \| \Gzw_K \|_{\mathcal{H}_2}^2 \right\},
\end{equation}
with the cost function being the squared $\mathcal{H}_2$-norm of the closed-loop transfer function 
$\Gzw_K (z) \Let \Ccl_K (zI - \Acl_K )^{-1} \Bw$, 
from $w$ to $z$~\cite[Sec.~21.5]{ZhouDoyleGlover1996}. 
This cost can be expressed in the frequency domain or, equivalently, in the time domain as the sum of the squared norms of the impulse response of the system. 
This equivalence is a consequence of Parseval's theorem~\cite[Thm.~8.16]{rudin_pma}.
In particular, we have
\begin{align}\label{eq:standard-H2-cost-markov}
    v^{\St}(K) = \| \Gzw_K \|_{\mathcal{H}_2}^2 \Let \frac{1}{2\pi} \int_{0}^{2\pi} \Trc \left( \Gzw_K (e^{j\omega})^*\; \Gzw_K (e^{j\omega}) \right) d\omega =
    \sum_{i=0}^\infty \| \Ccl_K \Acl_K^i \Bw \|_{F}^2.
\end{align}
The standard $\mathcal{H}_2$ cost can be alternatively expressed using the \emph{observability} Gramian of the pair $(\Acl_K, \Ccl_K)$, given by~\cite[Ch.~6.6]{Chen1999LinearSystem}  
\begin{align}\label{eq:gram_obs}
    \Wo_K \;\Let\; \ssum_{t=0}^\infty (\Acl_K^t)\tr \Ccl_K\tr \Ccl_K \Acl_K^t. 
\end{align}
Indeed, we have (see, e.g.,~\cite[Rem.~21.6]{ZhouDoyleGlover1996}) 
\begin{equation}\label{eq:standard-H2-cost}
   v^{\St}(K) = \langle I , M_K \rangle, \quad \text{where} \quad  M_K \Let \Bw\tr \Wo_K \Bw, \qquad \forall K\in\Ks.
\end{equation}
Above, $M_K$ is the \emph{compressed} observability Gramian. 
The preceding representation follows directly from the time-domain expression; see also Lemma~\ref{lem: general_H2_cost}. 
We note that, for each $K\in\Ks$, the Gramian~$\Wo_K$ is the unique positive semi-definite solution of the Lyapunov equation~\cite[Lem.~21.2]{ZhouDoyleGlover1996}
\begin{align}\label{eq:gram_obs_lyap}
   \LO_{\Acl_K}(\Wo_K ) + \Ccl_K\tr \Ccl_K = 0  \quad \text{where} \quad  \LO_{A}(\Wo) := A\tr \Wo A - \Wo.
\end{align}

We next establish smoothness and compute the gradient of the standard cost~$v^{\St}$. 
Our derivations are direct applications of differential calculus to matrices; see~\cite{MagnusNeudecker2019MDC}. 
Similar results are provided for the LQR problem in, e.g., \cite{FazelGeKakadeMesbahi2018,Fatkhullin2021GradientFeedback}.
For the purpose of generality, given a matrix $\Sigma \succeq 0$ and any stabilizing gain $K\in\Ks$, let us define the corresponding \emph{controllability} Gramian of the pair $(\Acl_K, \Bw \Sigma^{1/2})$ by~\cite[Lem.~21.2]{ZhouDoyleGlover1996}
\begin{align}\label{eq:contr_gram}
    \Wc_K(\Sigma) \Let \ssum_{t=0}^{\infty} \Acl_K^{t} \Bw  \Sigma \Bw\tr \big( \Acl_K^{t} \big)\tr. 
\end{align}
We notice that $\Wc_K(\Sigma)$ is the unique positive semi-definite solution of the Lyapunov equation~\cite[Lem.~21.2]{ZhouDoyleGlover1996}
\begin{align}\label{eq:Lyap_eq_control}
   \LC_{\Acl_K}\big(\Wc_K(\Sigma)\big) +  \Bw  \Sigma \Bw\tr = 0 \quad \text{where} \quad  \LC_{A}(Y) := A Y A\tr - Y.
\end{align}
In particular, let 
$\Wc_K^{\St} \Let \Wc_K(I)$, 
denote the controllability Gramian of the pair $(\Acl_K, \Bw)$. 

\begin{Lem}[Regularity of $v^{\St}$] \label{lem:standard_H2_grad}
The standard cost~$v^{\St}$ is real-analytic on $\Ks$, and
\begin{align}\label{eq:standard_grad}
    \nabla_K v^{\St}(K) = 2 S_K \Wc_K^{\St} \quad \text{where} \quad S_K \Let \Bu\tr \Wo_K \Acl_K + D\tr \Ccl_K, \quad \forall K\in\Ks. 
\end{align}
\end{Lem}
\begin{proof} 
Recall that $v^{\St}(K) = \langle I , M_K \rangle$ for $K\in\Ks$. 
We start with showing $v^{\St}$ is real-analytic on $\Ks$. 
Observe that the mapping $\Acl_K \mapsto \Wo_K$ admits the convergent series representation~\eqref{eq:gram_obs} for all Schur $\Acl_K$. 
Therefore, $\Wo_K$ depends real-analytically on the entries of $\Acl_K = A + \Bu K$ and hence on $K$ in the constraint set~$\Ks$. 
Now, observe that the maps $ \Wo_K \mapsto M_K= \Bw\tr \Wo_K \Bw$ and $M_K \mapsto v^{\St}(K) = \langle I , M_K \rangle$ are both linear. 
Thus, the map $K \mapsto v^{\St}(K)$ is real-analytic for all $K\in\Ks$.

We next derive the gradient of $v^{\St}$. 
First, observe that
\begin{align}\label{eq:f-diff-Xo}
    dv^{\St} = \langle I , dM_K \rangle 
    = \langle I , \Bw\tr (d\Wo_K) \Bw \rangle 
    = \langle \Bw \Bw\tr ,  d\Wo_K \rangle.
\end{align}
Recall now the Lyapunov operator $\LO_{\Acl_K}$ and its adjoint operator $\LC_{\Acl_K}$ on the space of $n$-by-$n$, real, symmetric matrices and notice the corresponding adjoint identity
\begin{align}\label{eq:adjoint-lyap-identity}
    \langle Y , \LO_{\Acl_K}(X) \rangle = \langle \LC_{\Acl_K}(Y) , X \rangle.
\end{align}
By assumption, we have $\LC_{\Acl_K}(\Wc_{K}^{\St})+\Bw \Bw\tr = 0$. 
Then, plugging in $\Bw \Bw\tr = - \LC_{\Acl_K}(\Wc_{K}^{\St})$ to \eqref{eq:f-diff-Xo} and using the identity~\eqref{eq:adjoint-lyap-identity}, we obtain
\begin{align}\label{eq:f-diff-Xo-1}
    dv^{\St} = \langle - \LC_{\Acl_K}(\Wc_{K}^{\St}) ,  d\Wo_K \rangle =
    \langle - \Wc_{K}^{\St} ,  \LO_{\Acl_K}(d\Wo_K) \rangle.
\end{align}
Now, observe that, using the Lyapunov equation~\eqref{eq:gram_obs_lyap}, 
we have
\begin{align*}
     & (d\Acl_K)\tr \Wo_K \Acl_K + \Acl_K\tr (d\Wo_K) \Acl_K +  \Acl_K\tr \Wo_K (d\Acl_K) - d\Wo_K \\
     & \qquad \qquad   +(d\Ccl_K)\tr \Ccl_K + \Ccl_K\tr (d\Ccl_K)  = 0,
\end{align*}
and hence
\begin{align*}
    \LO_{\Acl_K}(d\Wo_K) = - (d\Acl_K)\tr \Wo_K \Acl_K - \Acl_K\tr \Wo_{K} (d\Acl_K) -(d\Ccl_K)\tr \Ccl_K -\Ccl_K\tr (d\Ccl_K).
\end{align*}
Plugging in the latter equality in~\eqref{eq:f-diff-Xo-1} and using the symmetry of $\Wc_{K}^{\St}$, we arrive at 
\begin{align*}
    dv^{\St} = 2 \langle \Wc_{K}^{\St} ,  \Acl_K\tr \Wo_K (d\Acl_K) + \Ccl_K\tr (d\Ccl_K)  \rangle.
\end{align*}
Finally, we use the identities $d\Acl_K = \Bu (dK)$ and $d\Ccl_K = D (dK)$ to get
\begin{align*}
    dv^{\St} = \langle \Wc_{K}^{\St} ,  2(\Acl_K\tr \Wo_K \Bu+ \Ccl_K\tr D) (dK) \rangle 
    = \langle 2 S_K \Wc_{K}^{\St},  dK \rangle.
\end{align*}
By definition, we have $dv^{\St} = \langle \nabla_K v^{\St}(K),dK \rangle$ and the claim follows.
\end{proof}

\subsection{Properties of standard $\mathcal{H}_2$-optimal solution}

In this subsection, we discuss two important properties of solution of the standard $\mathcal{H}_2$ problem~\eqref{eq:standard_H2_problem}. 
In this regard, let us notice that non-emptiness of $\Ks$ is a sufficient condition for problem~\eqref{eq:standard_H2_problem} to have a finite optimal value~\cite[Thm.~4.8]{trentelman1995sampled}. 
Indeed, there exists a stabilizing LQ-optimal value matrix~$\Wo_{\mathrm{LQ}}\succeq0$, that is, the largest real symmetric solution of the generalized discrete algebraic Riccati equation (DARE) associated with
$(A,\Bu,C,D)$, such that~\cite[Sec.~3 and 4]{trentelman1995sampled}
\begin{equation}\label{eq:LQ_value_matrix}
    \langle I , \Bw^\top \Wo_{\mathrm{LQ}}\Bw \rangle =  \inf_{K\in\Ks}v^{\St}(K).
\end{equation}
However, for the optimal value to be attained by a stabilizing gain, extra conditions are required; we refer the reader to~\cite[Sec.~4]{trentelman1995sampled} for the detailed discussion on the necessary and sufficient conditions for the existence of a stabilizing optimal solution for the standard $\mathcal{H}_2$ synthesis. 
The following remark provides a set of sufficient conditions for which the regular LQ-optimal gain is also a solution of the standard-$\mathcal{H}_2$ synthesis problem. 

\begin{Rem}[Regular LQ-optimal and standard $\mathcal{H}_2$-optimal gains]\label{Rem:standrad_H2_optimal_suff}
Assume that the pair $(A,\Bu)$ is stabilizable, 
the pair $(A,C)$ is detectable, 
$C\tr D = 0$, and $D\tr D \succ 0$. 
In this case the solution to the regular LQR problem is also a solution of the standard-$\mathcal{H}_2$ synthesis problem~\cite[Sec.~3]{trentelman1995sampled}.
Precisely, the LQ-optimal feedback gain 
\begin{equation*}\label{eq:H2_optimal_gain}
    K_{\mathrm{LQ}} = -(D\tr D+\Bu\tr \Wo_{K_\mathrm{LQ}} \Bu)^{-1} \Bu\tr \Wo_{K_\mathrm{LQ}} A,
\end{equation*}
is a standard $\mathcal{H}_2$-optimal stabilizing gain, 
where $\Wo_{K_{\mathrm{LQ}}} = \Wo_{\mathrm{LQ}}$ is the unique positive semi-definite solution of the DARE
\begin{equation*}\label{eq:H2_optimal_Riccati}
    \LO_{A}(\Wo) + C\tr C - A\tr \Wo \Bu  (D\tr D+\Bu\tr \Wo \Bu)^{-1} \Bu\tr \Wo A  = 0.
\end{equation*}
Observe that the stabilizing LQ-optimal gain $K_{\mathrm{LQ}}$ is independent of the matrix~$E$.
\end{Rem}

Our next results establishes that stationarity is not only necessary, but also sufficient for optimality of a stabilizing gain for the standard $\mathcal{H}_2$ problem.  

\begin{Thm}[Stationarity and optimality for $v^{\St}$]\label{thm:standard_H2_optimal_grad} 
A stabilizing gain $K\opt\in\Ks$ is a solution of the standard $\mathcal{H}_2$ problem~\eqref{eq:standard_H2_problem} if and only if $\nabla_K v^{\St}(K\opt) = 0$. 
In fact, if $K\opt\in\Ks$ and $\nabla_K v^{\St}(K\opt) = 0$, then
\begin{equation}\label{eq:standard_H2_optimal_grad}
    \left\| \Gzw_K \right\|_{\mathcal{H}_2}^{2} - \left\| \Gzw_{K\opt} \right\|_{\mathcal{H}_2}^{2} = \left\| \Gzw_K - \Gzw_{K\opt} \right\|_{\mathcal{H}_2}^{2},\quad \forall K\in\Ks.
\end{equation}
\end{Thm}

\begin{proof}
By Lemma~\ref{lem:standard_H2_grad}, $v^{\St}$ is real-analytic and hence continuously differentiable on $\Ks$. 
The necessity of stationarity then simply follows first-order condition. 
We next focus on the sufficiency of first-order condition for optimality. 
Hence, let $K\opt\in\Ks$ be such that $\nabla_K v^{\St}(K\opt) = 0$. 
By Lemma~\ref{lem:standard_H2_grad}, we thus have
\begin{equation}\label{eq:stationary_H2_SXzero}
    S_{K\opt}\Wc^{\St}_{K\opt}=0.
\end{equation}
Then, since $\Acl_{K\opt}$ is Schur, we can use the
convergent series representation~\eqref{eq:contr_gram} of the Gramian~$\Wc^{\St}_{K\opt}$, to write
\begin{align*}
    0 = S_{K\opt}\Wc^{\St}_{K\opt}S_{K\opt}\tr 
      = \sum_{t=0}^{\infty} S_{K\opt} \Acl_{K\opt}^t \Bw \Bw\tr (\Acl_{K\opt}\tr )^t S_{K\opt}\tr 
      = \sum_{t=0}^{\infty} \left( S_{K\opt} \Acl_{K\opt}^t \Bw \right)  \left(S_{K\opt} \Acl_{K\opt}^t \Bw\right)\tr.
\end{align*}
Every summand on the right-hand side of the preceding equation is positive semidefinite. 
Hence, every summand must be zero, which yields
\begin{equation}
\label{eq:stationary_H2_reachable_residual}
    S_{K\opt}\Acl_{K\opt}^t\Bw=0, \quad  \forall t\in\N_0.
\end{equation}
Thus, although $S_{K\opt}$ need not vanish on the entire state space,
it vanishes on the closed-loop disturbance-reachable subspace. 

Now, let $K\in\Ks$ be arbitrary. 
Also, let $e_j \in\R^p$, $j\in\{1,\ldots,p\}$, denote the canonical basis vectors
of the disturbance space. 
We now consider the disturbance impulse-response with state-feedback gains $K\opt$ and $K$, one at a time for a fixed $j$. 
To that end, we consider the deterministic trajectory (i.e., with $w_t=0$ for all $t\in\N_0$) produced by $K\opt$ and $K$ with initial state $x_0 = \Bw e_j$, by defining the two state, input, and performance sequences
\begin{align*}
    x_{t,j}^{\star} &\Let \Acl_{K\opt}^t\, \Bw\, e_j, &
    u_{t,j}^{\star} &\Let K\opt\, x_{t,j}^{\star}, &
    z_{t,j}^{\star} &\Let \Ccl_{K\opt}\, x_{t,j}^{\star}, \\
    x_{t,j} &\Let \Acl_K^t\, \Bw\, e_j, & 
    u_{t,j} &\Let K\, x_{t,j}, & 
    z_{t,j} &\Let \Ccl_K \, x_{t,j}.
\end{align*}
We also introduce the trajectory differences
\begin{equation*}
    \delta x_{t,j} \Let x_{t,j}-x_{t,j}^{\star}, \quad
    \delta u_{t,j} \Let u_{t,j}-u_{t,j}^{\star}, \quad
    \delta z_{t,j} \Let z_{t,j}-z_{t,j}^{\star}.
\end{equation*}
Since both trajectories have the same initial condition, we have $\delta x_{0,j}=0$. 
Moreover, since both trajectories satisfy the original dynamics, we get
\begin{equation} \label{eq:stationary_H2_difference_dynamics}
    \delta x_{t+1,j} = A\, \delta x_{t,j} + B\, \delta u_{t,j}, \quad
    \delta z_{t,j} =  C\, \delta x_{t,j} + D\, \delta u_{t,j}.
\end{equation}

We next construct the costate along the trajectory produced by $K\opt$. 
Define 
$p^{\star}_{t,j} \Let \Wo_{K\opt}\, x_{t,j}^{\star}$. 
We now derive two identities satisfied by this sequence. 
Multiplying the Lyapunov equation~\eqref{eq:gram_obs_lyap} (with $K=K\opt$) by $x_{t,j}^{\star}$ gives 
\begin{equation*}
p^{\star}_{t,j} = \Wo_{K\opt}\, x_{t,j}^{\star} = \Acl_{K\opt}\tr \Wo_{K\opt} \Acl_{K\opt}\, x_{t,j}^{\star}  + \Ccl_{K\opt}\tr \Ccl_{K\opt} \, x_{t,j}^{\star}.  
\end{equation*}
Then, using $ x_{t+1,j}^{\star} = \Acl_{K\opt}\, x_{t,j}^{\star}$ and $ z_{t,j}^{\star} =  \Ccl_{K\opt} \, x_{t,j}^{\star}$, we arrive at 
$p^{\star}_{t,j} = \Acl_{K\opt}\tr\, p^{\star}_{t+1,j} + \Ccl_{K\opt}\tr z_{t,j}^{\star}$. 
Expanding $\Acl_{K\opt}=A+BK\opt$ and $\Ccl_{K\opt}=C+DK\opt$, we then have
\begin{equation*} 
    p^{\star}_{t,j} = A\tr p^{\star}_{t+1,j} + C\tr z_{t,j}^{\star} + K\opt\tr (B\tr p^{\star}_{t+1,j} + D\tr z_{t,j}^{\star}).
\end{equation*}
On the other hand, using again $ x_{t+1,j}^{\star} = \Acl_{K\opt}\, x_{t,j}^{\star}$ and $ z_{t,j}^{\star} =  \Ccl_{K\opt} \, x_{t,j}^{\star}$, we obtain
\begin{align*}
    B\tr p^{\star}_{t+1,j} + D\tr z_{t,j}^{\star}  
    & = B\tr \Wo_{K\opt}\, \Acl_{K\opt}\, x_{t,j}^{\star} + D\tr \Ccl_{K\opt}\, x_{t,j}^{\star} 
     \\ & = \big( B\tr \Wo_{K\opt}\,\Acl_{K\opt}+D\tr \Ccl_{K\opt} \big)\, x_{t,j}^{\star} \\
    & = S_{K\opt}\, x_{t,j}^{\star}
     = S_{K\opt}\, \Acl_{K\opt}^t\, \Bw\, e_j
     = 0.
\end{align*}
By \eqref{eq:stationary_H2_reachable_residual}, it then follows that
\begin{align}\label{eq:stationary_H2_control_costate}
    B\tr p^{\star}_{t+1,j} + D\tr z_{t,j}^{\star}  = 0,
\end{align}
and as a result,
\begin{equation} \label{eq:stationary_H2_openloop_costate}
    p^{\star}_{t,j} = A\tr p^{\star}_{t+1,j} +  C\tr z_{t,j}^{\star}.
\end{equation}

With the construction above, we can now show that the stationary performance sequence~$(z_{t,j}^{\star})_{t=0}^{\infty}$ and the feasible performance variation sequence~$(\delta z_{t,j}^{\star})_{t=0}^{\infty}$ are orthogonal. 
Indeed, using~\eqref{eq:stationary_H2_difference_dynamics}, we have
\begin{align*}
    (z_{t,j}^{\star})\tr \delta z_{t,j} = (z_{t,j}^{\star})\tr ( C\, \delta x_{t,j}+D\, \delta u_{t,j} ) 
    =  \big(C\tr z_{t,j}^{\star} \big)\tr \delta x_{t,j} + \big( D\tr z_{t,j}^{\star} \big)\tr \delta u_{t,j}.
\end{align*}
Then, using~\eqref{eq:stationary_H2_control_costate} and~\eqref{eq:stationary_H2_openloop_costate}, we obtain
\begin{align*}
    (z_{t,j}^{\star})\tr \delta z_{t,j} &= \big(p^{\star}_{t,j}-A\tr p^{\star}_{t+1,j}\big)\tr \delta x_{t,j} - \big( B\tr p^{\star}_{t+1,j} \big)\tr \delta u_{t,j} \\
    &=  (p^{\star}_{t,j})\tr \delta x_{t,j} - (p^{\star}_{t+1,j})\tr \big(
        A\, \delta x_{t,j}+B\, \delta u_{t,j} \big). 
\end{align*}
Using the difference dynamics~\eqref{eq:stationary_H2_difference_dynamics}, we hence arrive at
$(z_{t,j}^{\star})\tr \delta z_{t,j}  =  (p^{\star}_{t,j})\tr \delta x_{t,j} - (p^{\star}_{t+1,j})\tr \delta x_{t+1,j}$. 
Summing the preceding equation from $t=0$ to
$t=T$ gives
\begin{equation} \label{eq:stationary_H2_finite_orthogonality}
    \ssum_{t=0}^{T} (z_{t,j}^{\star})\tr \delta z_{t,j} = (p^{\star}_{0,j})\tr \delta x_{0,j} - (p^{\star}_{T+1,j})\tr \delta x_{T+1,j}.
\end{equation}
The first term on the right-hand side is zero since $\delta x_{0,j}=0$. 
Also, since both $\Acl_{K\opt}$ and $\Acl_K$ are Schur stable, $x_{t,j}^{\star}$ and $x_{t,j}$ converge exponentially to zero as $t\to\infty$.
This, in turn, implies that $\delta x_{t,j} = x_{t,j}-x_{t,j}^{\star}$ and 
$p^{\star}_{t,j}=\Wo_{K\opt}\, x_{t,j}^{\star}$ converges exponentially to zero as $t\to\infty$. 
As a result, we have 
$\lim_{T\to\infty} (p^{\star}_{T+1,j})\tr \delta x_{T+1,j} = 0$. 
Taking the limit in
\eqref{eq:stationary_H2_finite_orthogonality} therefore gives
\begin{equation} \label{eq:stationary_H2_orthogonality}
    \ssum_{t=0}^{\infty}  (z_{t,j}^{\star})\tr \delta z_{t,j}
    =0.
\end{equation}
That is, the performance trajectory generated by $K\opt$ is orthogonal to its difference from the trajectory generated by any other stabilizing $K$. 
Now, observe that since $z_{t,j}=z_{t,j}^{\star}+\delta z_{t,j}$, we have 
$\|z_{t,j}\|_2^2 = \|z_{t,j}^{\star}\|_2^2 + 2 (z_{t,j}^{\star})\tr \delta z_{t,j} + \|\delta z_{t,j}\|_2^2$. 
Then, summing over $t\in\N_0$ and $j\in\{1,\ldots,p\}$ and using \eqref{eq:stationary_H2_orthogonality}, we arrive at
\begin{align}\label{eq:stationary_H2_channel_pythagorean}
    \sum_{j=1}^{p}\sum_{t=0}^{\infty}\|z_{t,j}\|_2^2 - \sum_{j=1}^{p}\sum_{t=0}^{\infty}\|z_{t,j}^{\star}\|_2^2 
    = \sum_{j=1}^{p}\sum_{t=0}^{\infty}\|z_{t,j} - z_{t,j}^{\star}\|_2^2.
\end{align}
Now, observe that
\begin{align*}
    \sum_{j=1}^{p}\sum_{t=0}^{\infty}\|z_{t,j}\|_2^2 
    = \sum_{j=1}^{p}\sum_{t=0}^{\infty}\|\Ccl_K\Acl_K^t\Bw e_j\|_2^2
    = \sum_{t=0}^{\infty}\|\Ccl_K\Acl_K^t\Bw\|_F^2.
\end{align*}
Therefore, using the representation~\eqref{eq:standard-H2-cost-markov} of the standard $\mathcal{H}_2$ cost, the equality~\eqref{eq:stationary_H2_channel_pythagorean} can be written as 
$\left\| \Gzw_K \right\|_{\mathcal{H}_2}^{2} - \left\| \Gzw_{K\opt} \right\|_{\mathcal{H}_2}^{2} = \left\| \Gzw_K - \Gzw_{K\opt} \right\|_{\mathcal{H}_2}^{2} \geq 0$. 
Since $K\in\Ks$ was arbitrary, $K\opt$ is globally optimal.
This concludes the proof.
\end{proof} 

The proof of the preceding result is inspired by the completion-of-squares treatment of singular discrete-time LQR and $\mathcal{H}_2$ problems in~\cite{trentelman1995sampled}, 
combined with the Lyapunov policy-gradient representation used for modern LQR policy
optimization in, e.g.,~\cite{FazelGeKakadeMesbahi2018}. 
Let us note that the existing gradient-domination bounds in 
\cite[Lem.~3 and Cor.~4]{FazelGeKakadeMesbahi2018} and \cite[Thm.~3]{watanabe2026gradient}
imply global optimality of stationary policies under the regular LQR conditions in Remark~\ref{Rem:standrad_H2_optimal_suff} and when the disturbance controllability Gramian~$\Wc_{K\opt}^{\St}$ is \emph{nonsingular}. 
Theorem~\ref{thm:standard_H2_optimal_grad} however does not impose any conditions and, in particular, allows for $\Wc_{K\opt}^{\St}$ to be singular.  
The main additional observation is that stationarity of $K\opt$ implies that the residual $S_{K\opt}$ vanishes on the closed-loop disturbance-reachable subspace (that is, $\range(\Wc_{K\opt}^{\St})$), even when it need not vanish on the entire state space. 
This leads to the exact Pythagorean performance-difference identity~\eqref{eq:standard_H2_optimal_grad} without assuming a positive-definite controllability Gramian. 
Theorem~\ref{thm:standard_H2_optimal_grad} also provides us with the following instrumental property of standard $\mathcal{H}_2$-optimal solution.

\begin{Cor}[Minimality of $M_{K\opt}$]\label{cor:obs_gram_minimal}
Let the stabilizing gain $K\opt\in\Ks$ be a solution of the standard $\mathcal{H}_2$ problem~\eqref{eq:standard_H2_problem}.  
Then, 
$M_{K\opt} \preceq M_{K}$ for all $K\in\Ks$. 
\end{Cor}
\begin{proof}
Standard $\mathcal{H}_2$-optimality of $K\opt$ implies that 
$S_{K_\star}A_{K_\star}^{t}E=0$ for all $t\in\N_0$; see~\eqref{eq:stationary_H2_reachable_residual} in the proof of Theorem~\ref{thm:standard_H2_optimal_grad}. 
This also implies stationarity of $K_\star$ for the standard $\mathcal H_2$ cost with the modified disturbance matrix~$\tilde{E}=Ey$ for any $y\in\mathbb R^p$. 
Applying Theorem~\ref{thm:standard_H2_optimal_grad} to this modified system, we obtain
\[
 y^\top(M_K-M_{K_\star})y
 =\left\|(G_K-G_{K_\star})y\right\|_{\mathcal H_2}^{2}\geq0, \qquad \forall K\in\Ks.
\]
Since $y$ is arbitrary, $M_{K_\star}\preceq M_K$ for every $K\in\Ks$.
\end{proof}

The preceding result can be viewed as a matrix-valued strengthening of standard $\mathcal{H}_2$ optimality. Indeed, the latter directly gives only $\Trc(M_{K\opt})\leq \Trc(M_K)$,
whereas Corollary~\ref{cor:obs_gram_minimal} establishes the stronger Loewner-order relation $M_{K\opt}\preceq M_K$ for all $K\in\Ks$. 
In this regard, we note that the LQ-optimal value matrix $\Wo_{\mathrm{LQ}}$ in~\eqref{eq:LQ_value_matrix} provides the common lower bound $\Wo_{\mathrm{LQ}}\preceq\Wo_K$, which, after compression through the disturbance matrix $E$, gives $M_{\mathrm{LQ}}\Let E\tr\Wo_{\mathrm{LQ}}\, E \preceq M_K$ for all $K\in\Ks$. 
The attainment of the standard $\mathcal{H}_2$ optimum then implies $M_{K\opt}  =  M_{\mathrm{LQ}} = E\tr\Wo_{\mathrm{LQ}}\, E$, however, it does not require $\Wo_{K_\star}=\Wo_{\mathrm{LQ}}$.
When $E$ is rank deficient, the standard $\mathcal{H}_2$ cost only measures the closed-loop performance along state directions excited through $E$. 
Hence, an $\mathcal{H}_2$-optimal gain need not be LQ-optimal for every initial state and may therefore satisfy $\Wo_{K\opt}\neq\Wo_{\mathrm{LQ}}$. 
Let us illustrate this through an example.

\begin{Ex}[Singular disturbance controllability Gramian]
\label{ex:singular_optima}
Consider
\[
 A=\begin{bmatrix}0&0\\0&2\end{bmatrix},\qquad
 B=\begin{bmatrix}0\\1\end{bmatrix},\qquad
 E=\begin{bmatrix}1\\0\end{bmatrix},\qquad
 C=\begin{bmatrix}1&0\\0&1\end{bmatrix},\qquad 
 D=\begin{bmatrix}0\\0\end{bmatrix}.
\]
Let $K=[\,k_1\ k_2\,]$ and set $\beta=2+k_2$.
Then, $K\in\Ks$ if and only if $|\beta|<1$, and $v^{\St}(K)=1+\frac{k_1^2}{1-\beta^2}$. 
Thus, every $K_\beta=[\,0\ \beta-2\,]$ with $|\beta|<1$ is stationary and globally optimal for the standard $\mathcal{H}_2$ problem.
At these gains, $\Wc^{\St}_{K_\beta}=\operatorname{diag}(1,0)$, $\Wo_{K_\beta}=\operatorname{diag}(1,\frac{1}{1-\beta^2})$, and $S_{K_\beta}= [\,0\;\; \frac{\beta}{1-\beta^2}\,]$. 
In particular, $S_{K_\beta}\ne0$ for $\beta\ne0$.
The LQ-optimal value matrix is $\Wo_{\mathrm{LQ}}=I_2$, attained by the gain $K_{\mathrm{LQ}}=[\,0\;\; {-2}\,]$. 
For $\beta\ne0$, $\Wo_{K_\beta}\ne \Wo_{\mathrm{LQ}}$, however, $M_{K_\beta} = M_{K_{\mathrm{LQ}}} = 1$. 
This illustrates both stationarity sufficiency with a singular
disturbance Gramian and optimality only after compression.
\end{Ex}

\section{Distributionally robust $\mathcal{H}_2$ synthesis}
\label{sec:DR-H2}

In this section, we formulate and study distributionally robust (DR) $\mathcal{H}_2$ synthesis by considering i.i.d.\ zero-mean disturbances with a generic covariance. 
We again analyze the regularity properties of the corresponding cost and show the \emph{sufficiency} of the first-order stationarity condition for optimality.
More importantly, we show that the standard $\mathcal{H}_2$ synthesis is inherently distributionally robust against this class of disturbances. 
The analysis is in particular enabled by the covariance representation of the performance cost of $\mathcal{H}_2$~synthesis, combined with the Loewner minimality in
Corollary~\ref{cor:obs_gram_minimal}. 

\subsection{Generalized $\mathcal{H}_2$ cost with arbitrary disturbance covariance}

The standard $\mathcal{H}_2$ cost can be alternatively expressed as the steady-state power of the performance signal under zero initial state and i.i.d.~white noise disturbance with zero mean and unit covariance, that is, 
\begin{equation*}
    v^{\St}(K) =  \lim_{t\ra\infty} \EE \left[\| z_t \|_2^2 \;\middle|\; x_0 = 0,\; w_i\overset{\text{i.i.d}}{\sim}\mathcal{N}(0,I) \right], \quad \forall K\in\Ks.
\end{equation*}
We now generalize the standard cost above by allowing the disturbance to have a generic zero-mean distribution with finite covariance. 
To that end, let us define the corresponding set of distributions
\begin{equation*}
    \PS \Let \{\Pw \in \mathcal{B}(\R^p) \;:\; \mu_{\Pw} = 0,\; \|\Sp\|_{F}<\infty\},
\end{equation*}
and the generalized $\mathcal{H}_2$ cost functional
\begin{equation}\label{eq:general_H2_cost}
    v(K,\Pw) \;\Let\;  \lim_{t\ra\infty} \EE \left[\| z_t \|_2^2 \;\middle|\; x_0 = 0,\; w_i\overset{\text{i.i.d}}{\sim}\Pw \right], \quad \forall K\in\Ks,\, \forall P\in\PS.
\end{equation}
Similar to the standard cost, the generalized cost can be computed algebraically using the observability Gramian.

\begin{Lem}[Generalized cost $v$]\label{lem: general_H2_cost} 
$v(K, \Pw) = \langle \Sp , M_K \rangle$ for $K\in\Ks$ and $P\in\PS$. 
\end{Lem}
\begin{proof}
For the closed-loop system~\eqref{eq:closed-loop} with $x_0 = 0$ and driven by the disturbance $w_t$, we have 
$z_{t+1} = \ssum_{i=0}^t \Ccl_K \Acl_K^i \Bw w_i$ for all $t\in\N_0$. 
Then, we can write for all $t\geq 1$, 
\begin{align*}
    \EE \left[\| z_t \|_2^2 \;\middle|\; x_0 = 0 \right] = \EE \left[\left\| \ssum_{i=0}^{t-1} \Ccl_K \Acl_K^i \Bw w_i \right\|_2^2 \right].
\end{align*}
Now, we use the properties of the disturbance. 
First, using linearity of expectation and the fact that $w_i$'s are independent and zero-mean, we arrive at
\begin{align*}
   \EE \left[\left\| \ssum_{i=0}^{t-1} \Ccl_K \Acl_K^i \Bw w_i \right\|_2^2 \right] 
    = \sum_{i=0}^{t-1}  \EE \left[\left\| \Ccl_K \Acl_K^i \Bw w_i \right\|_2^2 \right].
\end{align*}
Let $H_{K,i} \Let \Bw\tr (\Acl_K^i)\tr \Ccl_K\tr \Ccl_K \Acl_K^i \Bw$. 
Then, using the fact that $w_i$'s are identically distributed with distribution $\Pw$, we have
\begin{align*}
    \ssum_{i=0}^{t-1} \EE \left[\left\| \Ccl_K  \Acl_K^i \Bw w_i \right\|_2^2 \right] &= \ssum_{i=0}^{t-1} \EE_{\Pw} \left[w_i\tr H_{K,i} w_i  \right] 
      = \ssum_{i=0}^{t-1} \EE_{\Pw} \left[\langle w_i w_i\tr , H_{K,i} \rangle\right] \\
    & =  \left\langle \Sp ,  \ssum_{i=0}^{t-1}H_{K,i} \right\rangle.
\end{align*}
Finally, using the definition of the cost in~\eqref{eq:general_H2_cost} and the Gramian in~\eqref{eq:gram_obs}, we obtain
\begin{align*}
    v(K,\Pw) = \lim_{t\ra\infty} \left\langle \Sp , \Bw\tr \left(\ssum_{i=0}^{t-1}(\Acl_K^i)\tr \Ccl_K\tr \Ccl_K \Acl_K^i\right) \Bw \right\rangle = \langle \Sp ,\Bw\tr \Wo_K \Bw \rangle.
\end{align*}
This concludes the proof.
\end{proof} 

Lemma~\ref{lem: general_H2_cost} reveals several key insights. 
Most importantly, by expressing the cost as the Frobenius inner product~$ \langle \Sp , M_K\rangle$, it separates the effect of disturbance statistics $\Sp$ from the dynamics encoded in $M_K$, which is a key element in the subsequent distributional robustness analysis. 
Moreover, it shows that the cost~$v(\cdot,\Pw)$ generalizes the standard cost $v^{\St}$ synthesis by allowing a generic covariance~$\Sp$ for the disturbance. 
In particular, given $\Pw\in\PS$, let $\tilde{v}^{\St}$ be the standard cost for the dynamics matrices $(A,B,C,D,\tilde{\Bw})$ with $\tilde{\Bw} = \Bw \Sp^{1/2}$, that is, 
$\tilde{v}^{\St}(K) = \ssum_{i=0}^\infty \| \Ccl_K \Acl_K^i \Bw \Sp^{1/2} \|_{F}^2$. 
Then, using the representation~\eqref{eq:standard-H2-cost} for the standard cost, we have
\begin{equation}\label{eq:standard_general_equiv}
    \begin{aligned}
    v(K,\Pw) & 
    = \tilde{v}^{\St}(K), \quad \forall K\in\Ks.
    \end{aligned}
\end{equation}
This, in turn, implies that the results of the previous section for the standard $\mathcal{H}_2$ cost and problem can be extended to the generalized $\mathcal{H}_2$ cost~\eqref{eq:general_H2_cost} and the corresponding problem~$\inf_{K\in\Ks} v(K,\Pw)$ for every $\Pw\in\PS$, by replacing the disturbance matrix~$\Bw$ with $\Bw \Sp^{1/2}$. 
In particular, we have the following results on smoothness and gradient of the generalized cost and the sufficiency of stationarity for optimality of stabilizing gains for the generalized cost. 

\begin{Cor}[Regularity of $v$] \label{cor:general_H2_grad}
Let $\Pw\in\PS$. 
The nominal cost~$v(\cdot,\Pw)$ is real-analytic on $\Ks$. Also, 
$\nabla_K v(K,\Pw) = 2 S_K \Wc_K^{\Pw}$ for all $K\in\Ks$, 
where $\Wc_K^{\Pw} \Let \Wc_K(\Sp)$ is the solution to the Lyapunov equation~\eqref{eq:Lyap_eq_control} with $\Sigma = \Sp$.
\end{Cor}
\begin{proof} 
The result follows from Lemma~\ref{lem:standard_H2_grad} and the equivalence~\eqref{eq:standard_general_equiv}.
\end{proof}

\begin{Cor}[Stationarity and optimality for $v$]\label{cor:general_H2_optimal_grad} 
Let $\Pw\in\PS$. 
We have $K\opt \in \argmin_{K\in\Ks} v(K,\Pw)$ if and only if $\nabla_K v(K\opt,\Pw) = 0$. 
In fact, if $K\opt\in\Ks$ and $\nabla_K v(K\opt,\Pw) = 0$, then 
\begin{equation*}\label{eq:gen_H2_optimal_grad}
     v(K,\Pw) -  v(K\opt,\Pw) = \left\| (\Gzw_K - \Gzw_{K\opt}) \Sp^{1/2} \right\|_{\mathcal{H}_2}^{2},\quad \forall K\in\Ks.
\end{equation*}
\end{Cor}
\begin{proof} 
The result follows from Theorem~\ref{thm:standard_H2_optimal_grad} and the equivalence~\eqref{eq:standard_general_equiv}.
\end{proof}

\subsection{DR $\mathcal{H}_2$ cost}

With the generalized $\mathcal{H}_2$ cost~\eqref{eq:general_H2_cost} at hand, we can now formulate the DR version of $\mathcal{H}_2$ problem. 
Given an ambiguity set $\D$, we define the DR  $\mathcal{H}_2$ synthesis by
\begin{equation}\label{eq:dr_H2_problem}
    v^{\DR}\opt := \inf_{K \in \Ks}\, \left\{ v^{\DR}(K) \Let \sup_{\Pw \in \D} v(K, \Pw) \right\}.
\end{equation}
Throughout this section, we assume the following.

\begin{As}[Ambiguity]\label{as:ambiguity}
    $\D$ is a non-empty subset of $\PS$ and independent of the gain~$K$. 
    Moreover, $\ \sup_{\Pw\in\D} \|\Sp\|_F < \infty$. 
\end{As}

We next consider the regularity properties of the DR cost $v^{\DR}$. 
To characterize the regularity of the DR cost, it is useful to consider
the set $\{\Sp \;:\; \Pw\in\D \}$ of covariance matrices generated by the ambiguity set, and its closed convex hull $\mathcal{C} \Let \overline{\operatorname{conv}} (\{\Sp \;:\; \Pw\in\D \})$. 
Under Assumption~\ref{as:ambiguity}, the set
$\mathcal{C}$ is a nonempty compact subset of the set of $p$-by-$p$ positive semi-definite matrices. 
Moreover, since the generalized cost is linear in the covariance
matrix, Lemma~\ref{lem: general_H2_cost} gives
\begin{align}    
    v^{\DR}(K) = \sup_{\Pw\in\D}\; \langle \Sp,M_K\rangle 
    = \sup_{\Sigma\in\{\Sp \;:\; \Pw\in\D \}}\; \langle \Sigma,M_K\rangle 
    = \max_{\Sigma\in\mathcal{C}}\; \langle \Sigma,M_K\rangle, 
    \quad \forall K\in\Ks. \label{eq:DR_cost_covariance_representation}
\end{align}
Thus, although the supremum in the definition of $v^{\DR}$
need not be attained by a distribution in $\D$, the equivalent
maximization problem~\eqref{eq:DR_cost_covariance_representation}
is always attained at the covariance level in $\mathcal{C}$. 
For $K\in\Ks$, define the corresponding set of active worst-case covariances by 
$\mathcal{C}_{K} \Let \argmax_{\Sigma\in\mathcal{C}}\; \langle \Sigma,M_K\rangle$. 
The set $\mathcal{C}_{K}$ is a nonempty compact convex set, namely, an exposed face of $\mathcal{C}$. 
We can now characterize the regularity and the Clarke subdifferential of the DR cost. 
We recall that $ \Wc_K(\Sigma)$ is the solution to the Lyapunov equation~\eqref{eq:Lyap_eq_control}.

\begin{Prop}[Regularity of $v^{\DR}$]\label{prop:DR_H2_Clarke}
Let Assumption~\ref{as:ambiguity} hold. 
The DR cost $v^{\DR}$ is finite, locally Lipschitz continuous, and Clarke regular on $\Ks$. 
Moreover, the Clarke subdifferential of $v^{\DR}$ is
\begin{equation}\label{eq:DR_Clarke_subdiff}
    \partial_{\rm C}v^{\DR}(K) = \left\{ 2\;S_K\;\Wc_K(\Sigma_{K}) \;:\; \Sigma_{K}\in\mathcal{C}_{K} \right\},  \quad \forall K\in\Ks.
\end{equation}
In particular, if $\mathcal{C}_{K}=\{\Sigma_K\}$ is a singleton, then $v^{\DR}$ is differentiable at $K$ and 
$\nabla_K v^{\DR}(K) = 2\;S_K\;\Wc_K(\Sigma_K)$. 
\end{Prop}

\begin{proof}
We first establish the representation of the DR cost as a pointwise
maximum of smooth functions.
For every $\Sigma\in\mathcal{C}$, define 
$v_{\Sigma}(K) \Let  v\big(K, \mathcal{N}(0,\Sigma)\big) = \langle \Sigma,M_K\rangle$ for $K\in\Ks$. 
By Corollary~\ref{cor:general_H2_grad}, $v_{\Sigma}$ is real-analytic on $\Ks$ and $\nabla_K v_{\Sigma}(K) = 2S_K\Wc_K(\Sigma)$. 
Moreover, notice that the index set $\mathcal{C}$ is compact and independent of $K$,
and $(K,\Sigma)\mapsto v_{\Sigma}(K)$ and
$(K,\Sigma)\mapsto\nabla_Kv_{\Sigma}(K)$ are continuous on
$\Ks\times\mathcal{C}$.
By the equality ~\eqref{eq:DR_cost_covariance_representation}, we have 
$ v^{\DR}(K) = \max_{\Sigma\in\mathcal{C}}\; v_{\Sigma}(K)$. 
That is, the DR cost is a pointwise maximum of smooth functions. 
Let us first verify local Lipschitz continuity. 
Set $r_{\D} \Let \max_{\Sigma\in\mathcal{C}}\|\Sigma\|_F <\infty$. 
Then, for every $K_1,K_2\in\Ks$, we have 
\begin{align*}
    |v^{\DR}(K_1)-v^{\DR}(K_2)| \;\leq\; \max_{\Sigma\in\mathcal{C}}\; \left| \langle\Sigma , M_{K_1} -M_{K_2} \rangle \right| 
    \;\leq\; r_{\D}\,\|M_{K_1}-M_{K_2}\|_F.
\end{align*}
Recall that $K\mapsto M_K$ is real-analytic and hence locally
Lipschitz on $\Ks$. Thus, $v^{\DR}$ is locally Lipschitz on $\Ks$. 
We next apply the classical result on pointwise maxima of smooth functions; see, e.g.,
\cite[Thm.~2.8.2]{clarke1990optimization} and \cite[Thm.~10.31]{rockafellar1998variational}.
Since the maximizing set $\mathcal{C}_K$ of $\max_{\Sigma\in\mathcal{C}} v_{\Sigma}(K)$ is compact 
and the parametrization is continuously differentiable in $K$,
$v^{\DR}$ is Clarke regular and 
$\partial_{\rm C}v^{\DR}(K) = \operatorname{conv} \{ \nabla_Kv_{\Sigma_{K}}(K) = 2\;S_K\;\Wc_K(\Sigma_{K}) \;:\; \Sigma_{K}\in\mathcal{C}_{K} \}$. 
Notice that the map $\Sigma_{K}\mapsto\Wc_K(\Sigma_{K})$ is linear, which follows from the convergent series expression~\eqref{eq:contr_gram} of the controllability Gramian. 
Therefore, the map $\Sigma_{K} \mapsto 2S_K\Wc_K(\Sigma_{K})$ is also linear. 
Then, since $\mathcal{C}_{K}$ is convex, its image under a linear map is already convex.
Therefore, the convex hull in the preceding expression is redundant, and we obtain ~\eqref{eq:DR_Clarke_subdiff}. 
Finally, if the active covariance is unique,
$\mathcal{C}_{K}=\{\Sigma_K\}$, the Clarke
subdifferential is the singleton $\partial_{\rm C}v^{\DR}(K) = \left\{ 2S_K\Wc_K(\Sigma_K) \right\}$. 
The pointwise-maximum differentiation result then gives
$\nabla_K v^{\DR}(K) = 2\;S_K\;\Wc_K(\Sigma_K)$.
\end{proof}

\subsection{Properties of DR $\mathcal{H}_2$-optimal solution}

In this subsection, we extend the result of Theorem~\ref{thm:standard_H2_optimal_grad} for standard $\mathcal{H}_2$ problem to the DR $\mathcal{H}_2$ problem~\eqref{eq:dr_H2_problem}. 
To that end, observe that Proposition~\ref{prop:DR_H2_Clarke} allows us to define stationarity for the possibly nonsmooth DR cost in the standard sense of Clarke. 
Namely, a stabilizing gain $K\in\Ks$ is Clarke stationary if $0\in\partial_{\rm C}v^{\DR}(K)$. 
In general nonsmooth and nonconvex optimization, Clarke stationarity is only a necessary condition for local optimality and is not sufficient for optimality. 
However, similar to the standard $\mathcal{H}_2$ problem, stationarity is also sufficient for global optimality for the considered DR $\mathcal{H}_2$ problem.

\begin{Thm}[Stationarity and optimality for $v^{\DR}$] \label{thm:DR_H2_optimal_grad}
Let Assumption~\ref{as:ambiguity} hold. 
Then, a stabilizing gain
$K\opt\in\Ks$ is a solution of the DR $\mathcal{H}_2$
problem~\eqref{eq:dr_H2_problem} if and only if $K\opt$ is Clarke stationary, that is, $0\in\partial_{\rm C}v^{\DR}(K\opt)$. 
In fact, if $K\opt\in\Ks$ is Clarke stationary, then there exists an active covariance
$\Sigma_{K\opt}\in\mathcal{C}_{K\opt}$ such that 
\begin{equation}\label{eq:DR_performance_difference_lower_bound}
    v^{\DR}(K)-v^{\DR}(K\opt) \geq \left\| (\Gzw_K-\Gzw_{K\opt})\Sigma_{K\opt}^{1/2} \right\|_{\mathcal{H}_2}^{2}, \quad \forall K\in\Ks.
\end{equation}
\end{Thm}

\begin{proof}
We prove the two implications separately. 
Suppose first that $v^{\DR}\opt = v^{\DR}(K\opt)$. 
Since the stabilizing set $\Ks$ is open, $K\opt$ is an unconstrained local minimizer of $v^{\DR}$ on an open neighborhood contained in $\Ks$. 
By Proposition~\ref{prop:DR_H2_Clarke}, $v^{\DR}$ is locally Lipschitz on $\Ks$. 
The Clarke necessary optimality condition \cite[Prop.~2.3.2]{clarke1990optimization} therefore yields Clarke stationarity, that is, $0\in\partial_{\rm C}v^{\DR}(K\opt)$. This proves necessity.

We next prove sufficiency. Let $K\opt\in\Ks$ satisfy $0\in\partial_{\rm C}v^{\DR}(K\opt)$. 
By the exact subdifferential representation~\eqref{eq:DR_Clarke_subdiff}, there exists $\Sigma_{K\opt}\in\mathcal{C}_{K\opt}$ such that $2S_{K\opt}\Wc_{K\opt} (\Sigma_{K\opt})=0$. 
Now, let $\Pw_{K\opt} \Let \mathcal{N}(0,\Sigma_{K\opt})\in\PS$ and consider the generalized $\mathcal{H}_2$ cost associated with $\Pw_{K\opt}$, that is, $v(K, \Pw_{K\opt}) = \langle\Sigma_{K\opt},M_K\rangle$. 
By Corollary~\ref{cor:general_H2_grad}, $\nabla_K v(K\opt, \Pw_{K\opt}) = 2\, S_{K\opt}\, \Wc_{K\opt}(\Sigma_{K\opt}) = 0$. 
We can therefore apply
Corollary~\ref{cor:general_H2_optimal_grad} to get
\begin{equation}\label{eq:effective_covariance_perf_diff}
    v(K, \Pw_{K\opt}) - v(K\opt, \Pw_{K\opt}) =
    \left\| (\Gzw_K-\Gzw_{K\opt})\Sigma_{K\opt}^{1/2} \right\|_{\mathcal{H}_2}^{2},
    \quad  \forall K\in\Ks.
\end{equation}
On the other hand, since
$\Sigma_{K\opt}\in\mathcal{C}_{K\opt}$, we have
\begin{equation}\label{eq:effective_covariance_active}
    v^{\DR}(K\opt) \;=\; \langle\Sigma_{K\opt},M_{K\opt}\rangle \;=\; v(K\opt, \Pw_{K\opt}).
\end{equation}
Furthermore, for an arbitrary $K\in\Ks$, it follows that 
\begin{equation}\label{eq:DR_dominates_effective}
    v^{\DR}(K) \;=\; \max_{\Sigma\in\mathcal{C}} \langle \Sigma,M_K\rangle
    \;\geq\; \langle\Sigma_{K\opt},M_K\rangle 
    = v(K, \Pw_{K\opt}).
\end{equation}
Combining~\eqref{eq:effective_covariance_perf_diff}, \eqref{eq:effective_covariance_active}, and
\eqref{eq:DR_dominates_effective}, we obtain \eqref{eq:DR_performance_difference_lower_bound}.
\end{proof}

Theorem~\ref{thm:DR_H2_optimal_grad} provides an interpretation of Clarke stationarity in terms of an active worst-case covariance.
Indeed, $0\in\partial_{\rm C}v^{\DR}(K\opt)$ if and only if there exists $\Sigma_{K\opt}\in\mathcal{C}_{K\opt}$ such that $K\opt$ is stationary for the generalized cost $v\big(\cdot, \mathcal{N}(0,\Sigma_{K\opt})\big)$.
The special landscape of the $\mathcal{H}_2$ problem then implies
that $K\opt$ is globally optimal for this generalized cost.
Since $\Sigma_{K\opt}$ is simultaneously active for the DR cost at
$K\opt$, this nominal global optimality transfers directly to the DR
problem. 
We can construct a \emph{supporting-plane-type} lower bound for the DR cost in the form of $v^{\DR}(K) \geq \langle \Sigma_{K\opt}, M_K \rangle$ for all $K \in \Ks$, with the bound being tight at $K=K\opt$; 
see Figure~\ref{fig:K_space_bound} for a graphical illustration. 

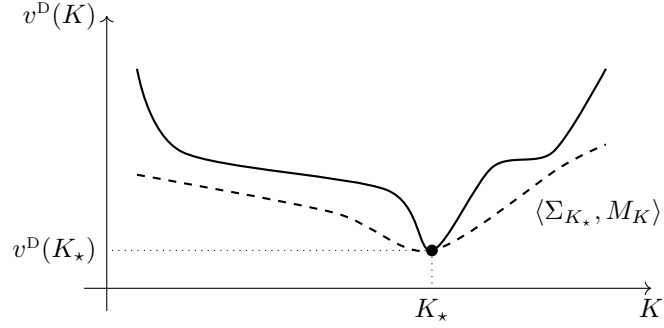
\begin{figure}[t]
\centering
\begin{tikzpicture}[
    scale=0.8,
    arr/.style={-{Latex[length=3mm]},thick},
    pt/.style={circle,fill,inner sep=1.4pt},
    every node/.style={font=\scriptsize}
]

\draw[->] (-0.3,0) -- (7.2,0) node[below] {$K$};
\draw[->] (0,-0.3) -- (0,3.6) node[left] {$ v^{\DR}(K)$};
\draw[dotted] (4.3,0.5) -- (0,0.5) node[left] {$v^{\DR}(K\opt)$};
\draw[dotted] (4.3,0.5) -- (4.3,0) node[below] {$K\opt$};

\coordinate (Keq) at (4.3,0.5);
\filldraw[black] (Keq) circle (2pt);

\draw[thick]
    plot[smooth]
    coordinates{
        (0.4,2.9)
        (1,1.8)
        (3.7,1.3)
        (Keq)
        (5.1,1.6)
        (5.9,1.8)
        (6.6,2.9)
    };

\draw[thick,dashed]
    plot[smooth]
    coordinates{
        (0.4,1.5)
        (3,1)
        (Keq)
        (6,1.6)
        (6.6,1.9)
    };

\node[align=right] at (6.5,1) {$\langle \Sigma_{K\opt} , M_K \rangle$};

\end{tikzpicture}
\caption{An intuitive presentation of Theorem~\ref{thm:DR_H2_optimal_grad}. $K\opt\in\Ks$ is a Clarke stationary point of the DR $\mathcal{H}_2$ cost $v^{\DR}(K) = \max_{\Sigma\in\mathcal{C}}\; \langle \Sigma,M_K\rangle$, and $\Sigma_{K\opt}\in\argmax_{\Sigma\in\mathcal{C}}\; \langle \Sigma,M_{K\opt}\rangle$ is a corresponding active covariance.}
\label{fig:K_space_bound}
\end{figure}

It is important that $\Sigma_{K\opt}$ need not, under Assumption~\ref{as:ambiguity} alone, be the covariance of an actual
worst-case distribution in $\D$. 
Rather, it belongs to the closed convex hull of the covariance matrices generated by $\D$. 
No attainment assumption on the original distributional maximization is therefore required. 
If, in addition, $\{\Sp \;:\; \Pw\in\D \}$ is closed and convex, then $\mathcal{C}=\{\Sp \;:\; \Pw\in\D \}$ and the active covariance $\Sigma_{K\opt}$ is the covariance of an actual distribution $\Pw_{K\opt}\in\D$. 
In that case, $\Pw_{K\opt}$ can be chosen to be a worst-case distribution at $K\opt$, and $\nabla_Kv(K\opt,\Pw_{K\opt})=0$.

\subsection{Distributional robustness of standard $\mathcal{H}_2$ synthesis}

We now focus on the connection between the solutions of the standard and DR $\mathcal{H}_2$ problems by showing the inherent distributionally robustness of standard $\mathcal{H}_2$ synthesis. 
The following result is a direct consequence of the characterization of the generalized cost in Lemma~\ref{lem: general_H2_cost} and the minimality of the compressed observability Gramian~$M_{K\opt}$ in Loewner order established in Corollary~\ref{cor:obs_gram_minimal}. 

\begin{Thm}[From standard to DR $\mathcal{H}_2$ optimality] \label{thm:dr_H2_standard} 
Let $K\opt\in\Ks$ be a solution of the standard $\mathcal{H}_2$ problem~\eqref{eq:standard_H2_problem}. Then, 
$v(K\opt, \Pw) \leq v(K, \Pw)$ for all $K\in\Ks$ and $\Pw\in\PS$. 
In particular, under Assumption~\ref{as:ambiguity}, we have $v^{\DR}\opt = v^{\DR}(K\opt)$, that is, the gain~$K\opt$ is also a solution of the DR $\mathcal{H}_2$ problem~\eqref{eq:dr_H2_problem}.    
\end{Thm}

\begin{proof}
By Corollary~\ref{cor:obs_gram_minimal}, for the standard $\mathcal{H}_2$-optimal gain $K\opt$, we have $M_K-M_{K\opt}\succeq0$ for all $K\in\Ks$. 
On the other hand, every distribution $\Pw\in\PS$
has a positive semi-definite covariance matrix
$\Sp\succeq0$.
As a result,
\begin{align*}
    \langle \Sp,M_K-M_{K\opt} \rangle = \langle \Sp,M_K \rangle - \langle \Sp,M_{K\opt} \rangle  \geq 0, \quad \forall K\in\Ks,\; \forall \Pw\in\PS.
\end{align*}
By Lemma~\ref{lem: general_H2_cost}, $v(K,\Pw)=\langle\Sp,M_K\rangle$ and hence we arrive at $v(K\opt, \Pw) \leq v(K, \Pw)$. 
Since the latter inequality holds for every
$\Pw\in\D \subset \Pw$, taking the supremum over the same ambiguity set
$\D$ on both sides yields
\begin{align*}
     v^{\DR}(K\opt) = \sup_{\Pw\in\D}v(K\opt,\Pw) \;\leq\; \sup_{\Pw\in\D}v(K,\Pw) = v^{\DR}(K),    \quad \forall K\in\Ks. 
\end{align*}
The DR cost $v^{\DR}$ is finite under Assumption~\ref{as:ambiguity}. 
Therefore, the preceding inequality shows that $K\opt \in \argmin_{K\in\Ks}v^{\DR}(K)$, which completes the proof.
\end{proof}

The preceding result indicates that a solution $K\opt\in\Ks$ to the standard $\mathcal{H}_2$ problem~\eqref{eq:standard_H2_problem} can be used for all the ambiguity sets satisfying Assumption~\ref{as:ambiguity}. 
The ambiguity set changes its certified worst-case performance, even though a standard $\mathcal{H}_2$-optimal gain remains optimal. 
The theorem gives an inclusion of optimizer sets; it does not assert uniqueness or equality of the two optimizer sets. 
In Appendix~\ref{app:invisible}, we describe cost-preserving gain perturbations that can explain nonuniqueness of standard $\mathcal{H}_2$-optimal gains.
Moreover, in Appendix~\ref{app:DR_stand_optimality}, we provide a sufficient
condition for a DR-optimal gain to also be standard $\mathcal{H}_2$-optimal.

Recall that under the regular assumptions of Remark~\ref{Rem:standrad_H2_optimal_suff}, a common optimal controller is already provided by classical LQR theory.
In particular, the stabilizing LQ-optimal gain $K_{\mathrm{LQ}}$ depends only on $(A,B,C,D)$ and not on $\Bw$ or the disturbance covariance. 
Most importantly, its value matrix satisfies $\Wo_{K_{\mathrm{LQ}}} = \Wo_{\mathrm{LQ}}\preceq \Wo_K$ for all $K\in\Ks$, and hence $\langle\Sp,\Bw\tr \Wo_{K_{\mathrm{LQ}}}\Bw\rangle
 \leq\langle\Sp,\Bw\tr \Wo_K \Bw\rangle$ for every $\Sp\succeq0$.
That is, $K_{\mathrm{LQ}}$ is optimal for every covariance, and consequently for every ambiguity set considered here, as also noted in~\cite{Taskesen2023DRLQ}. 
This classic inherent distributional robustness of $K_{\mathrm{LQ}}$ can also be established using the gradient dominance property of the LQR objective studied in~\cite{FazelGeKakadeMesbahi2018, Bu2019LQRFirstOrder, watanabe2026gradient}. 
Theorem~\ref{thm:dr_H2_standard} however applies more generally to \emph{every} attained standard $\mathcal H_2$ optimum, without imposing any particular assumption on the system matrices and, in particular, without requiring a nonsingular disturbance controllability Gramian~$\Wc_{K\opt}^{\St}$.
As discussed above, a standard $\mathcal H_2$-optimal gain $K\opt$ can satisfy $\Wo_{K_\star}\ne \Wo_{\mathrm{LQ}}$, where $\Wo_{\mathrm{LQ}}$ is the LQ-optimal
value matrix in~\eqref{eq:LQ_value_matrix}, while $\Bw\tr \Wo_{K_\star}\Bw=\Bw\tr \Wo_{\mathrm{LQ}}\Bw$. 
Theorem~\ref{thm:dr_H2_standard} shows that this compressed equality, rather than equality of the full value matrices, is sufficient for the DR optimality.

We finish with a simple application of the preceding result for the ``spectral'' ambiguity set characterized by an upper bound $\Sigma_0 \succeq 0$ on the disturbance covariance.

\begin{Prop}[Spectral DR $\mathcal{H}_2$]\label{prop: DR cost spec}
Consider the DR $\mathcal{H}_2$ problem~\eqref{eq:dr_H2_problem} with 
\begin{equation}\label{eq:dr_spec_cov}
    \D = \left\{ \Pw \in \PS \;:\; \Sp \preceq \Sigma_0 \right\}.
\end{equation}
where $\Sigma_0 \succeq 0$. 
Then, $v^{\DR} (K) = \langle \Sigma_0 , M_K \rangle$ for all $K\in\Ks$. 
Moreover, if $K\opt\in\Ks$ is a standard $\mathcal{H}_2$-optimal solution, 
then $v^{\DR}\opt = \langle \Sigma_0 , M_{K\opt} \rangle$ 
and any $\Pw\opt \in \D$ with covariance $\Sigma_{\Pw\opt}=\Sigma_0$ is a worst-case distribution in~\eqref{eq:dr_H2_problem}.
\end{Prop}
\begin{proof}
Fix $K\in\Ks$ and recall that $M_K \succeq 0$. 
By Lemma~\ref{lem: general_H2_cost}, we have 
$v^{\DR} (K) = \sup_{\Pw \in \D} v(K, \Pw) = \max_{0\preceq \Sp \preceq \Sigma_0} \langle \Sp , M_K \rangle$. 
Let $\Sp$ be any feasible covariance and define $\Delta := \Sigma_0 - \Sp \succeq 0$. 
Observe that 
$\langle \Sigma_0 , M_K \rangle - \langle \Sp , M_K \rangle = \langle \Sigma_0 - \Sp , M_K \rangle = \langle \Delta , M_K \rangle \geq 0$, 
where we used the fact that $\Delta\succeq 0$ and $M_K\succeq 0$ for the last inequality.
Hence, $\langle \Sp , M_K \rangle \le \langle \Sigma_0 , M_K \rangle$ for every feasible $\Sp$, and the maximum is attained at $\Sigma_{\Pw\opt}=\Sigma_0$. 
Therefore, $v^{\DR} (K) = \langle \Sigma_0 , M_K \rangle$ for all $K\in\Ks$. 
The result then follows from Theorem~\ref{thm:dr_H2_standard}. 
\end{proof}

The preceding result indicates that the worst-case distribution is independent of the feedback gain for the spectral ambiguity characterized by~\eqref{eq:dr_spec_cov}. 
This is because $M_K \succeq 0$ for any gain~$K$ and hence the functional $\Sp \mapsto v(K,\Sp) = \langle \Sp , M_K \rangle$ is monotone with respect to the Loewner order.
As a result, the inner maximization in DR problem~\eqref{eq:dr_H2_problem} always selects the extreme point $\Sigma_{\Pw\opt} = \Sigma_0$, which is independent of $K$; see Figure~\ref{fig:ambiguity_sets} for a graphical illustration. 
The simplification observed here is not universal and may break down for other ambiguity sets (or uncertainty models).
In such cases, the adversary reacts to the controller through $M_K$ and the worst-case distribution depends on $K$. 
This dependence affects performance evaluation, while Theorem~\ref{thm:dr_H2_standard} still provides a common optimal gain.
The next section derives explicit worst-case performance certificates 
and identifies the corresponding covariances for three cases of ambiguity sets.

\section{Worst-case performance certificates for specific ambiguity sets}
\label{sec:specific-ambiguity-set}

Once a standard $\mathcal H_2$-optimal stabilizing gain is available, Theorem~\ref{thm:dr_H2_standard} guarantees the DR optimality of this gain for ambiguity sets satisfying Assumption~\ref{as:ambiguity}. 
The remaining task is then to evaluate its worst-case performance. 
In this section, we look at three cases of ambiguity sets for which the worst-case distribution depends on the (optimal) feedback gain. 
To be precise, we consider ambiguity sets based on the Frobenius ball around a nominal covariance and KL and W2 balls around a nominal zero-mean Gaussian distribution. 
The presented results quantify the performance guarantee as the ambiguity radius changes and identify adverse disturbance directions. 
In order to exclude the degenerate case of 
$v(K\opt,\Pw) = 0$ for all $\Pw\in\PS$, 
we consider the following assumption throughout this section. 

\begin{As}[Non-zero $M_{K\opt}$]\label{as:non-zero-M_K}
The standard $\mathcal{H}_2$ problem~\eqref{eq:standard_H2_problem} has a solution $K\opt$ such that $M_{K\opt}\neq 0$.  
\end{As}

Let us first remark that Assumption~\ref{as:non-zero-M_K} implies that 
$\Trc(M_{K}) = v^{\St}(K) \geq v^{\St}(K\opt) = \Trc(M_{K\opt}) > 0$, 
and hence
$M_{K}\neq 0$ for all $K\in\Ks$. 

\subsection{Frobenius ambiguity set} 
Let us define a moment-based ambiguity set as the Frobenius ball of radius $\epsilon \geq 0$ around a nominal covariance $\Sigma_0 \succeq 0$, that is,
\begin{subequations}\label{eq:dr_frob}
\begin{equation}\label{eq:frob_set}
   \D^{\F}(\epsilon,\Sigma_0) \;\Let\; \left\{ \Pw \in \PS \;:\; \left\|\Sp- \Sigma_0 \right\|_{F} \leq \epsilon \right\},
\end{equation}
and the corresponding DR $\mathcal{H}_2$ synthesis problem
\begin{align}\label{eq:dr_frob_prob}
    v^{\F}\opt := \min_{K \in \Ks}\, \left\{ v^{\F}(K) \Let \sup_{\Pw \in \D^{\F}(\epsilon,\Sigma_0)} v(K, \Pw) \right\}.
\end{align}
\end{subequations}

Our first result formalizes the Frobenius DR cost~$v^{\F}$ and the corresponding worst-case distribution as a function of the gain $K$.

\begin{Prop}[Frobenius DR cost $v^{\F}$]\label{prop: DR cost frob}
Let Assumption~\ref{as:non-zero-M_K} hold. 
Then, $v^{\F}$ is real-analytic on $\Ks$ and given by 
\begin{align}\label{eq:frobenius-DR-cost}
        v^{\F}(K) = \langle \Sigma_K^{\F} , M_K \rangle \quad \text{where} \quad \Sigma_K^{\F} \Let \Sigma_0 + \frac{\epsilon}{\|M_K\|_{F}} M_K, \qquad \forall K\in\Ks.
\end{align}
In particular, any $\Pw^{\F}_K \in \D^{\F}(\epsilon,\Sigma_0)$ with covariance $\Sigma_{\Pw^{\F}_K}=\Sigma_K^{\F}$ is a worst-case distribution in~\eqref{eq:dr_frob_prob}, corresponding to the gain~$K\in\Ks$.
\end{Prop}

\begin{proof}
We first derive the characterization~\eqref{eq:frobenius-DR-cost}. 
Observe that the set of covariance matrices generated by the ambiguity set $\D^{\F}(\epsilon,\Sigma_0)$ is the closed and convex set 
$\mathcal{C}^{\F}(\epsilon,\Sigma_0) = \{ \Sigma  \;:\; \Sigma \succeq 0,\; \| \Sigma - \Sigma_0 \|_{F} \leq \epsilon \}$. 
Then, using the representation~\eqref{eq:DR_cost_covariance_representation} of the DR $\mathcal{H}_2$ cost, for any fixed $K \in \Ks$, 
$v^{\F}(K) =  \max_{\Sigma \in \mathcal{C}^{\F}(\epsilon,\Sigma_0) } \; \langle \Sigma , M_K \rangle$. 
Using the Cauchy-Schwarz inequality for the Frobenius inner product, the objective function in the preceding maximization can be upper bounded as follows
\begin{align*}
    \langle \Sigma , M_K \rangle &= \langle \Sigma + \Sigma_0 - \Sigma_0 , M_K \rangle = \langle  \Sigma_0 , M_K \rangle + \langle \Sigma - \Sigma_0 , M_K \rangle \\
    & \leq \langle  \Sigma_0 , M_K \rangle + \|\Sigma - \Sigma_0\|_{F} \cdot \|M_K\|_{F}.
\end{align*}
Therefore, for any feasible covariance $\Sigma $ with $\|\Sigma - \Sigma_0 \|_{F} \leq \epsilon$, we have 
$\langle \Sigma , M_K \rangle \leq \langle \Sigma_0 , M_K \rangle + \epsilon \cdot \|M_K\|_{F}$. 
By Assumption~\ref{as:non-zero-M_K}, we have $M_K\neq 0$ and hence $\Sigma_K^{\F}$ is well-defined. 
Now, observe that $\Sigma_K^{\F}\in \mathcal{C}^{\F}(\epsilon,\Sigma_0)$ (since $\Sigma_0, M_K \succeq 0$ and $\| \Sigma_K^{\F}- \Sigma_0 \|_{F} = \epsilon$) and the upper bound in the preceding equation is attained by $ \Sigma = \Sigma_K^{\F}$. 

We next show that $v^{\F}$ is real-analytic on $\Ks$. 
We have previously shown that the map $K\mapsto M_K$ is real-analytic on $\Ks$;  see the proof of Lemma~\ref{lem:standard_H2_grad}. 
Then, since $M_K \neq 0$ for all $K\in\Ks$ by Assumption~\ref{as:non-zero-M_K}, the function $v^{\F}$ is real-analytic on $\Ks$.
\end{proof} 

For the Frobenius ambiguity set~\eqref{eq:frob_set}, there is a nonlinear coupling between the controller gain and the disturbance covariance in the DR cost $v^{\F}$. 
In particular, the worst-case covariance $\Sigma_K^{\F}$ is now a function of the gain $K$, breaking down the separation enjoyed in the spectral DR $\mathcal{H}_2$ framework; see Figure~\ref{fig:ambiguity_sets} for a graphical illustration. 
A similar coupling between the controller gain and the disturbance covariance will arise in the next two cases of ambiguity sets.

\begin{figure}[t]
\centering
\small
\begin{tikzpicture}[scale=0.7, every node/.style={font=\scriptsize}]

\begin{scope}[shift={(1,0)}]
    \fill[gray!30] (0.5,0.5) rectangle (3,3);
    \draw[thick] (2.5,4) -- (4,1);
    \draw[->,thick] (3.5,2) -- (4.5,2.5) node[right] {$M_K$};
    \draw[dashed,thick] (1,4) -- (4,2.5);
    \draw[->,dashed,thick] (3.5,2.75) -- (4,3.75) node[right] {$M_{K'}$};
    \filldraw[black] (3,3) circle (2pt);
    \filldraw[black] (0.5,0.5) circle (2pt);
    \node at (2.25,-0.5) {Spectral set: $0 \preceq \Sigma \preceq \Sigma_0$};
    \node[below left] at (.5,.5) {$0$};
    \node[below left] at (3,3) {$\Sigma_0$};
\end{scope}

\begin{scope}[shift={(7.5,-0.2)}]

    \fill[gray!30] (2,2) circle (1.4);
    \node[below left] at (2,2) {$\Sigma_0$};
    \draw[dotted, thick, black!70!black] (2,2) -- ++(1.0,-1.0) node[pos=.6, left, black!70!black] {$\epsilon$};
    \draw[thick] (2,2) -- (4,3);
    \draw[->,thick] (2,2) -- (3,2.5) node[below right] {$M_K$};
    \draw[dashed,thick] (2,2) -- (3,4);
    \draw[->,dashed,thick] (2,2) -- (2.5,3) node[above left] {$M_{K'}$};
    \filldraw[black] (2+0.62,2+1.25) circle (2pt);
    \filldraw[black] (2+1.27,2+0.62) circle (2pt);
    \filldraw[black] (2,2) circle (2pt);
    \node at (2.5,-0.3) {Frobenius ball: $\|\Sigma - \Sigma_0\|_{F} \leq \epsilon$};
\end{scope}

\end{tikzpicture}
\caption{Possible dependence of the worst-case covariance on the gain. 
The gray areas show the feasible covariances for the two ambiguity sets. The arrows indicate the direction of the cost gradient $M_K = \Bw\tr \Wo_K \Bw \succeq 0$, for two gains $K$ and $K'$.
(Left) For the spectral set~\eqref{eq:dr_spec_cov}, a worst-case covariance lies at the ``upper right corner", i.e., $\Sigma_0$, regardless of the gradient ($M_K$ vs $M_{K'}$) and hence the gain.
(Right) For the Frobenius ball in~\eqref{eq:frob_set}, the worst-case covariance lies at a point at the boundary of the ball, depending on the gradient and hence the gain. }
\label{fig:ambiguity_sets}
\end{figure}
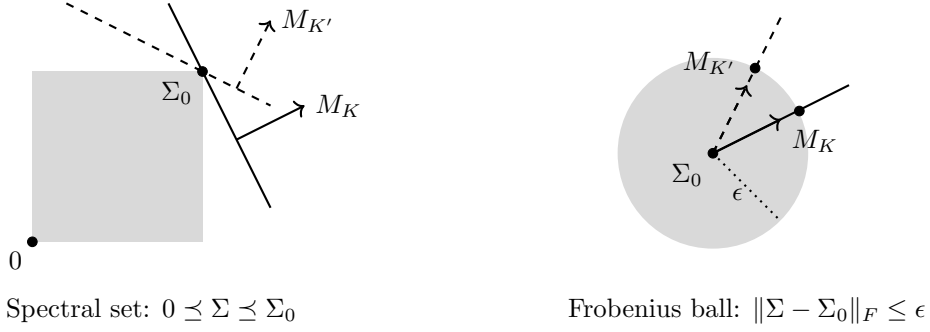

\subsection{KL ambiguity set} 
We now consider a model of distributional uncertainty based on the
KL divergence. 
Let us first introduce the KL divergence. 
For two probability distributions $\Qw$ and $\Pw$ on $\R^p$, provided $\Qw \ll \Pw$, that is, $\Qw$ is absolutely continuous with respect to $\Pw$, the KL divergence reads as
\begin{align*}
    \text{dist}^{\KL}(\Qw \| \Pw) := \int_{\R^p} \log\!\Big(\frac{d\Qw}{d\Pw}(w)\Big) d\Qw(w),
\end{align*}
where $d\Qw/d\Pw$ is the Radon-Nikodym derivative; see~\cite{kullback1951information} and
\cite[Sec.~1.4]{dupuis1997weak}. 
By convention, we have $\text{dist}^{\KL}(\Qw \| \Pw)=+\infty$ if $\Qw\not\ll \Pw$. 
Given a nominal zero-mean Gaussian distribution $\Pw_0 = \mathcal{N}(0,\Sigma_0)$ with $\Sigma_0\succeq 0$ and a radius $\epsilon>0$, let us consider the \emph{KL ball} 
\begin{subequations}\label{eq:dr_KL}
\begin{align}\label{eq:dr_KL_set}
    \D^{\KL}(\epsilon,\Sigma_0) \Let \big\{\Pw \in \PS \;:\; \text{dist}^{\KL}\big(\Pw \| \Pw_0\big) \leq \epsilon \big\},
\end{align}
as the ambiguity set and define the corresponding DR $\mathcal{H}_2$ synthesis problem
\begin{align}\label{eq:dr_KL_prob}
    v^{\KL}\opt := \min_{K \in \Ks}\, \left\{ v^{\KL}(K) \Let \sup_{\Pw \in \D^{\KL}(\epsilon,\Sigma_0)} v(K, \Pw) \right\}.
\end{align}
\end{subequations}
We next provide an explicit formula for the DR cost $v^{\KL}$ assuming $\Sigma_0 \succ 0$. 
This is done by using the duality theory and the Donsker-Varadhan variational formula~\cite[Lem.~3.1]{ahmadi2012entropic} to reduce the inner maximization problem in~\eqref{eq:dr_KL_prob} to a \emph{one-dimensional} convex minimization. 
In particular, this ambiguity set converts distributional robustness into an \emph{entropic penalty}~\cite[Thm.~3.3]{ahmadi2012entropic}. 
To that end, let us define 
\begin{equation*}
    \begin{array}{ll}
        \lambda^{\KL}_K \Let \lambda_{\max}(\Sigma_0^{1/2}\, M_K\, \Sigma_0^{1/2}), & \forall K\in\Ks; \\
        \phi^{\KL}(K,\theta) \Let  \epsilon \theta  - \frac{\theta}{2} \log \det \big(I - \frac{2}{\theta}\, \Sigma_0^{1/2}\, M_K\, \Sigma_0^{1/2} \big), & \forall K\in\Ks,\; \forall \theta\in(2\lambda^{\KL}_K, \infty).
    \end{array}
\end{equation*}
We have the following result for the preceding function.

\begin{Lem}[Unique minimizer of $\phi^{\KL}$]\label{lem: phi KL}
Let $\Sigma_0 \succ 0$ and Assumption~\ref{as:non-zero-M_K} hold. 
For each $K\in\Ks$, the function $\phi^{\KL}(K,\cdot)$ is \emph{strictly convex} and \emph{coercive} over the domain $\theta > 2\lambda^{\KL}_K$, and has a unique minimizer
$\theta^{\KL}_K \Let \argmin_{\theta} \{ \phi^{\KL}(K,\theta)  \;:\: \theta > 2\lambda^{\KL}_K \}$. 
\end{Lem}

\begin{proof}
Let $\{\lambda_i\}_{i=1}^p$ be the eigenvalues of the matrix product $\Sigma_0^{1/2} M_K \Sigma_0^{1/2} \succeq 0$ (recall that $\Sigma_0\succ 0$ by assumption and $M_K\succeq 0 $ by construction). 
In particular, since $M_K \neq 0$ by Assumption~\ref{as:non-zero-M_K}, we have that $\lambda^{\KL}_K = \max_{i\in\{1,\ldots,p\}} \lambda_i > 0$. 
Then, 
$\phi(K,\theta) = \epsilon \theta  - \frac{\theta}{2} \ssum_{i=1}^p  \log\big(1 - \frac{2 \lambda_i}{\theta} \big)$, 
and differentiating twice gives
\begin{align*}
    \frac{\partial^2}{\partial \theta^2} \phi(K,\theta) = \sum_{i=1}^p \frac{2 \lambda_i^2}{\theta (\theta - 2\lambda_i)^2} \geq \frac{2 (\lambda^{\KL}_K)^2}{\theta (\theta - 2\lambda^{\KL}_K)^2} > 0, \quad \forall \theta \in (2\lambda^{\KL}_K, \infty).
\end{align*}
This shows that $\phi(K,\cdot)$ is strictly convex for $\theta > 2\lambda^{\KL}_K$. 
As for coercivity, we clearly have $\phi(K,\theta)\uparrow \infty$ as $\theta\uparrow\infty$ since $\epsilon > 0$ by definition. 
On the other hand, observe that when $\theta \downarrow \big(2\lambda^{\KL}_K\big)$, the matrix~$\big(I - \frac{2}{\theta}\Sigma_0^{1/2} M_K \Sigma_0^{1/2}\big)$ loses positive definiteness with its log-det decreasing unboundedly and hence $\phi(K,\theta)\uparrow \infty$. 
This completes the proof.   
\end{proof}

With the preceding result in hand, we can now provide the following formula.

\begin{Prop}[KL DR cost $v^{\KL}$]\label{prop: DR cost KL}
Let $\Sigma_0 \succ 0$ and Assumption~\ref{as:non-zero-M_K} hold.
Then, $v^{\KL}$ is real-analytic on $\Ks$ and given by 
\begin{align}\label{eq:KL-DR-cost}
    v^{\KL}(K) = \langle \Sigma^{\KL}_K  , M_K \rangle \quad \text{where} \quad \Sigma^{\KL}_K \Let \Big(\Sigma_0^{-1} - \frac{2}{\theta^{\KL}_K} M_K \Big)^{-1} ,\qquad \forall K\in\Ks.
\end{align}
In particular, the worst-case distribution in~\eqref{eq:dr_KL_prob} corresponding to the gain~$K\in\Ks$ is the zero-mean Gaussian $\Pw^{\KL}_K = \mathcal{N}(0,\Sigma_{\Pw^{\KL}_K})$ with covariance $\Sigma_{\Pw^{\KL}_K} = \Sigma^{\KL}_K$. 
\end{Prop}

\begin{proof}
We first derive the characterization~\eqref{eq:KL-DR-cost}. 
Fix $K\in\Ks$. By Lemma~\ref{lem: general_H2_cost}, the generalized cost can be equivalently written as
\begin{align}\label{eq:dr_kl_primal_proof}
    v(K, \Pw) = \langle \Sp , M_K \rangle = \langle \EE_{\Pw}[ww\tr] , M_K \rangle =  \EE_{\Pw}[ \langle ww\tr , M_K \rangle ] = \EE_{\Pw}[w\tr M_K w].
\end{align}
Let us then define the performance map $g_K:\mathbb{R}^p\to\mathbb{R}$ by 
$g_K(w) := w\tr M_K w$, 
so that KL~DR cost~$v^{\KL}$ can be formulated as 
\begin{align}\label{eq:v_KL-proof}
    v^{\KL} (K) = \sup_{\Pw \in \D^{\KL}(\epsilon,\Sigma_0)}  \EE_{\Pw}[g_K(w)].
\end{align}
In what follows, we are going to expand the ambiguity set (by removing the restrictions on $\Pw$ to be zero-mean and have finite covariance) and consider the problem 
\begin{align}\label{eq:v_KL-proof-relaxed}
    \ol{v}^{\KL} (K) \Let \sup_{\Pw}  \big\{\EE_{\Pw}[g_K(w)] \;:\; \Pw \in \mathcal{B}(\R^p),\; \text{dist}^{\KL}(\Pw \| \Pw_0) \leq \epsilon  \big\}.
\end{align}
As we shall see, the optimal solution of the preceding ``relaxed'' problem~\eqref{eq:v_KL-proof-relaxed} is a zero-mean Gaussian distribution belonging to the set $\D^{\KL}(\epsilon,\Sigma_0) \subseteq \PS$. 
Therefore, the optimal solution to the original problem~\eqref{eq:v_KL-proof} is the same zero-mean Gaussian distribution. 
The problem~\eqref{eq:v_KL-proof-relaxed} is a well-studied case in DRO. 
Let us define the domain
\begin{subequations}\label{eq:KL_robust_cost_gen}
\begin{align}\label{eq:KL_robust_cost_domain}
    \Theta_K := \Big\{ \theta > 0  \;:\;  \EE_{\Pw_0} \big[e^{g_K(w)/\theta}\big] < \infty \Big\} \subset \R.
\end{align} 
If $\Theta_K \neq \varnothing$, we then have the dual representation \cite[Thm.~1]{hu2013kullback} 
\begin{align}\label{eq:KL_robust_cost_opt}
    \ol{v}^{\KL} (K) = \inf_{ \theta \in \Theta_K } \Big\{ \epsilon \theta  + \theta \log \EE_{\Pw_0} \big[e^{g_K(w)/\theta}\big] \Big\}.
\end{align}
\end{subequations}
Now, observe that since $\Pw_0=\mathcal{N}(0,\Sigma_0)$ with $\Sigma_0\succ 0$ by assumption and $M_K$ is symmetric (and positive semi-definite) for $K\in\Ks$, we can use the results for the moment generating function of quadratic forms of normal random variables (see, e.g., \cite[Cor.~3.2a1]{MathaiProvost1992}) to write 
\begin{align*}
     \EE_{\Pw_0} \big[e^{g_K(w)/\theta}\big] = \left\{\begin{array}{ll}
         \det\big(I - \frac{2}{\theta} \Sigma_0^{1/2} M_K \Sigma_0^{1/2} \big)^{-1/2} & \text{if } \theta > 2\lambda^{\KL}_K,  \\
         \infty & \text{otherwise}.
     \end{array} \right.
\end{align*}
Then, using the fact that
\begin{align*}
    \log\det\big(I - \frac{2}{\theta} \Sigma_0^{1/2} M_K \Sigma_0^{1/2}\big)^{-1/2} = - \frac{1}{2} \log\det\big(I - \frac{2}{\theta} \Sigma_0^{1/2} M_K \Sigma_0^{1/2}\big),
\end{align*}
we have 
$\ol{v}^{\KL}(K) = \inf_{\theta} \big\{ \phi^{\KL}(K,\theta) \;:\: \theta > 2\lambda^{\KL}_K \big\}$. 
By Lemma~\ref{lem: phi KL}, the preceding problem admits a unique dual minimizer $\theta^{\KL}_K$. 
Moreover, the corresponding worst-case distribution~$\Pw^{\KL}_K$ in~\eqref{eq:v_KL-proof-relaxed} is the exponential tilt of $\Pw_0$ characterized by the Radon-Nikodym derivative (see the proof of \cite[Thm.~2]{hu2013kullback})
\begin{align*}
    \frac{d\Pw^{\KL}_K}{d\Pw_0}(w) = \frac{e^{g_K(w)/\theta^{\KL}_K}}{\EE_{\Pw_0} \big[e^{g_K(w)/\theta^{\KL}_K}\big]}.
\end{align*}
For the nominal distribution $\Pw_0 = \mathcal{N}(0,\Sigma_0)$, the corresponding density is given by 
\begin{align*}
    p_0(w) = (2\pi)^{-p/2}\, \det(\Sigma_0)^{-1/2}\, \exp\Big(-\frac{1}{2} w\tr \Sigma_0^{-1} w \Big).
\end{align*}
Tilting by $\exp \big(g_K(w) / \theta^{\KL}_K \big) $ then gives the unnormalized worst-case density
\begin{align*}
    p^{\KL}_K(w) & \propto\exp \big(g_K(w) / \theta^{\KL}_K \big)\, p_0(w) 
     \propto \exp \Big( - \frac{1}{2} w\tr \big( \Sigma_0^{-1} - \frac{2}{\theta^{\KL}_K} M_K \big) w \Big) \\*
    & = \exp \Big( - \frac{1}{2} w\tr (\Sigma^{\KL}_{K})^{-1} w \Big).
\end{align*}
In particular, by Lemma~~\ref{lem: phi KL}, $\theta^{\KL}_K > 2\lambda^{\KL}_K$, and hence 
\begin{align*}
    (\Sigma^{\KL}_{K})^{-1}  = \Sigma_0^{-1/2} \Big( I - \frac{2}{\theta^{\KL}_K} \Sigma_0^{1/2} M_K \Sigma_0^{1/2} \Big) \Sigma_0^{-1/2} \succ 0.
\end{align*}
Therefore, the normalized tilted distribution, that is, the corresponding worst-case distribution for the relaxed problem~\eqref{eq:v_KL-proof-relaxed}, is the zero-mean Gaussian distribution~$\Pw^{\KL}_K = \mathcal{N}(0,\Sigma_{\Pw^{\KL}_K})$ with $\Sigma_{\Pw^{\KL}_K} = \Sigma^{\KL}_{K}$. 
This, in turn, implies that $\Pw^{\KL}_K$ is also the optimal solution of the original problem~\eqref{eq:v_KL-proof}.  
Then, using the primal moment formulation~\eqref{eq:dr_kl_primal_proof}, 
$v^{\KL}(K) = v(K,\Pw^{\KL}_K) =  \langle \Sigma^{\KL}_{K} , M_K\rangle$. 

We next show that $v^{\KL}$ is real-analytic on $\Ks$. 
The main ingredient is showing that the map $K\ra \theta^{\KL}_K$ is real-analytic on $\Ks$. 
To that end, let us denote $R_K \Let \Sigma_0^{1/2}M_K\Sigma_0^{1/2}$ and define the map
\begin{align*}
   \Phi^{\KL}(K,\theta) \Let \frac{\partial}{\partial\theta} \phi^{\KL}(K,\theta) = \epsilon - \frac{1}{2} \log \det \big( I - \frac{2}{\theta} R_K \big) - \frac{1}{\theta}
      \Big\langle \big( I - \frac{2}{\theta} R_K \big)^{-1} ,R_K \Big\rangle, 
\end{align*}
over the domain $(K,\theta) \in \Ks \times (2\lambda^{\KL}_K, \infty)$. 
In particular, observe that $\big( I - \frac{2}{\theta} R_K \big) \succ 0$ over this domain since $\lambda^{\KL}_K = \lambda_{\max}(R_K)$. 
Also, recall that $M_K$ is a real-analytic function over $K\in\Ks$; see the proof of Lemma~\ref{lem:standard_H2_grad}. 
Then, since matrix inversion and log-det are real-analytic on the space of positive definite matrices, the map $(K,\theta)\mapsto\Phi^{\KL}(K,\theta)$ is real-analytic on its domain. 

Now, fix $K\in\Ks$. 
By Lemma~\ref{lem: phi KL}, $\theta^{\KL}_K$ is the unique minimizer of the strictly convex and coercive function~$\phi^{\KL}(K,\cdot)$, and hence, we have $\Phi^{\KL}(K,\theta^{\KL}_K) = 0$ and $\frac{\partial}{\partial\theta} \Phi^{\KL}(K,\theta^{\KL}_K) > 0$. 
Therefore, by the real-analytic implicit function theorem (see, e.g., \cite[Thm.~2.4.4]{krantz2002implicit}), there exists a neighborhood $\mathcal{K} \subset \Ks$, an interval $\mathcal{I}$ containing $\theta^{\KL}_K$, and a real-analytic map $\psi_K:\mathcal{K}\ra\mathcal{I}$ such that 
$\Phi^{\KL}\big(\tilde{K},\psi_K(\tilde{K})\big) = 0$ for all $\tilde{K}\in\mathcal{K}$. 
Moreover, since $\phi^{\KL}(K,\cdot)$ is strictly convex and coercive and hence has unique global minimizer, the equation $\Phi^{\KL}(\tilde{K}, \theta) = 0$ has exactly one solution, namely, $\theta_{\tilde{K}}$. 
Therefore, the implicit solution of $\Phi^{\KL}\big(\tilde{K},\psi_K(\tilde{K})\big) = 0$ must coincide with this global minimizer, that is, $\psi_K(\tilde{K}) = \theta_{\tilde{K}}$ for all $\tilde{K}\in\mathcal{K}$. 
Thus, $\theta_{\tilde{K}}$ is a real-analytic function of $\tilde{K}$ in a neighborhood of every point $K\in\Ks$. 
Hence, it follows that the map $K\mapsto\theta^{\KL}_K$ is real-analytic over the entire open set~$\Ks$ of stabilizing controllers.     

Finally, recall that $v^{\KL}(K) = \langle \Sigma^{\KL}_K , M_K \rangle$, where 
$\Sigma^{\KL}_K = \Sigma_0^{1/2} ( I - \frac{2}{\theta^{\KL}_K} R_K )^{-1} \Sigma_0^{1/2}$. 
Also, $M_K$ and $\theta^{\KL}_K$ are real-analytic on $\Ks$ and $(I - \frac{2}{\theta^{\KL}_K} R_K) \succ 0$ for all $K\in\Ks$. 
Since matrix inversion is real-analytic on the space of positive definite matrices, $\Sigma^{\KL}_K$ is also a real-analytic function of $K\in\Ks$. 
Hence, $v^{\KL}$ is real-analytic on $\Ks$. 
\end{proof}


\subsection{W2 ambiguity set} 
We now focus on ambiguity sets defined based on the W2 distance. 
Let us first recall the definition of the Wasserstein distance. 
For two distributions $\Qw$ and $\Pw$ on $\R^p$, let $\Pi(\Qw,\Pw)$ denote the set of all probability measures $\pi$ on $\R^p\times\R^p$ with marginals $\Qw$ and $\Pw$  (a.k.a.~couplings or transport plans), that is,
\begin{align*}
    \pi(A\times\R^p)=\Qw(A) \quad \text{~and~} \quad
    \pi(\R^p\times B)=\Pw(B),
    \quad\text{for all Borel~}A,B\subset\R^p.
\end{align*}
The Wasserstein-2 (W2) distance is defined by (see~\cite[Defs.~6.1 and~6.4]{villani2009optimal})
\begin{equation*}
    \text{dist}^{\WA}(\Qw,\Pw) \Let \left(
    \inf_{\pi\in\Pi(\Qw,\Pw)}\,
    \int_{\R^p\times\R^p}\|w-\tilde w\|_2^2\,d\pi(w,\tilde w)
    \right)^{\!1/2}.
\end{equation*}
Notice that $\text{dist}^{\WA}(\Qw,\Pw)<\infty$ if $\Qw$ and $\Pw$ have a finite second moment. 
Given the nominal zero-mean Gaussian distribution $\Pw_0 = \mathcal{N}(0,\Sigma_0)$ with $\Sigma_0\succeq 0$ and a radius $\epsilon>0$, we now consider the \emph{W2 ball} 
\begin{subequations}\label{eq:dr_W2}
\begin{align}\label{eq:dr_W2_set}
    \D^{\WA}(\epsilon,\Sigma_0) \Let \{\Pw \in \PS \;:\; \text{dist}^{\WA}(\Pw,\Pw_0) \leq \epsilon\},
\end{align}
as the ambiguity set and define the corresponding DR $\mathcal{H}_2$ synthesis problem
\begin{align}\label{eq:dr_W2_prob}
    v^{\WA}\opt := \min_{K \in \Ks}\, \left\{ v^{\WA}(K) \Let \sup_{\Pw \in \D^{\WA}(\epsilon,\Sigma_0)} v(K, \Pw) \right\}.
\end{align}
\end{subequations}
We next provide an explicit formula for the W2 DR cost~$v^{\WA}$. 
Similar to the KL case discussed above, this is done by using the duality theory (to be precise, Kantorovich duality for quadratic transport cost~\cite[Ch.~1]{villani2021topics}). 
Thus, let us define 
\begin{equation*}
    \begin{array}{ll}
        \lambda^{\WA}_K \Let \lambda_{\max}(M_K), & \forall K\in\Ks; \\
        \phi^{\WA}(K,\theta) \Let \epsilon^2 \theta   + \Trc\!\left(\theta M_K(\theta I-M_K)^{-1}\Sigma_0\right), & \forall K\in\Ks,\; \forall\theta \in (\lambda^{\WA}_K, \infty).
    \end{array}
\end{equation*}
We have the following result for the preceding function.

\begin{Lem}[Unique minimizer of $\phi^{\WA}$]\label{lem: phi W2}
Let $\Sigma_0 \succ 0$ and Assumption~\ref{as:non-zero-M_K} hold. 
For each $K\in\Ks$, the function $\phi^{\WA}(K,\cdot)$ is \emph{strictly convex} and \emph{coercive} over the domain $\theta > \lambda^{\WA}_K$, and has a unique minimizer 
$\theta^{\WA}_K \Let \argmin_{\theta} \{ \phi^{\WA}(K,\theta)  \;:\: \theta > \lambda^{\WA}_K \}$. 
\end{Lem}
\begin{proof}
Let $\{\lambda_i\}_{i=1}^p$ be the eigenvalues of $ M_K \succeq 0$ (recall that $M_K$ is a symmetric and positive semi-definite by construction). 
In particular, notice that since $M_K \neq 0$ by Assumption~\ref{as:non-zero-M_K}, we have that $\lambda^{\WA}_K = \max_{i\in\{1,\ldots,p\}} \lambda_i > 0$. 
Diagonalize $M_K=U\operatorname{diag}(\lambda_1,\ldots,\lambda_p)U\tr$ and denote the $i$-th diagonal entry of $U\tr \Sigma_0 U$ by
$\mu_i$. 
Notice that since $\Sigma_0\succ0$ by assumption, each diagonal entry
$\mu_i$ is strictly positive. 
Also, let $\mu^{\WA}_K = \mu_j > 0$ for $j \in \argmax_i \lambda_i$.  
Then, it follows that 
$\phi^{\WA}(K,\theta) = \epsilon^2  \theta+ \sum_{i=1}^p \frac{ \mu_i\lambda_i\theta}{\theta-\lambda_i}$, 
and differentiating twice gives
\begin{align*}
    \frac{\partial^2}{\partial\theta^2}\phi^{\WA}(K,\theta)
    = \sum_{i=1}^p \frac{2\mu_i\lambda_i^2}{(\theta-\lambda_i)^3} 
    \geq  \frac{2\mu^{\WA}_K(\lambda^{\WA}_K)^2}{(\theta-\lambda^{\WA}_K)^3}  >0, \quad \forall \theta\in(\lambda^{\WA}_K, \infty). 
\end{align*}
Hence, $\phi^{\WA}(K,\cdot)$ is strictly convex over the domain $\theta>\lambda^{\WA}_K$.
Also, $\phi^{\WA}(K,\theta)\uparrow\infty$ as
$\theta\downarrow\lambda^{\WA}_K$ and as $\theta\uparrow\infty$; that is, $\phi^{\WA}(K,\cdot)$ is coercive on its domain. 
Thus, the minimizer $\theta^{\WA}_K$ is unique. 
\end{proof}
Using the preceding definitions, we can now provide the W2 DR cost explicitly.

\begin{Prop}[W2 DR cost $v^{\WA}$]\label{prop: DR cost W2}
Let $\Sigma_0 \succ 0$ and Assumption~\ref{as:non-zero-M_K} hold. 
Then, $v^{\WA}$ is real-analytic on $\Ks$ and given by
\begin{subequations}\label{eq:DR-W2-cost}
\begin{align}\label{eq:DR-W2-cost-only}
    v^{\WA}(K) = \langle \Sigma^{\WA}_K  , M_K \rangle, \quad \forall K\in\Ks,
\end{align}
where
\begin{align}\label{eq:DR-W2-covariance}
    \Sigma^{\WA}_K \Let (\theta^{\WA}_K)^2\,(\theta^{\WA}_K I-M_K)^{-1}\,\Sigma_0\,(\theta^{\WA}_K I-M_K)^{-1}.
\end{align}
\end{subequations}
In particular, the worst-case distribution in~\eqref{eq:dr_W2_prob} corresponding to the gain~$K\in\Ks$ is the zero-mean Gaussian $\Pw^{\WA}_K = \mathcal{N}(0,\Sigma_{\Pw^{\WA}_K})$ with covariance $\Sigma_{\Pw^{\WA}_K} = \Sigma^{\WA}_K$. 
\end{Prop}
\begin{proof} 
We first derive the characterization~\eqref{eq:DR-W2-cost}. 
Fix $K\in\Ks$. 
Recall that, using Lemma~\ref{lem: general_H2_cost} and the quadratic performance map $g_K:\mathbb{R}^p\to\mathbb{R}: w\mapsto w\tr M_K w $, the $\mathrm{W_2}$~DR cost~$v^{\WA}$ can be rewritten as
\begin{align}\label{eq:v_W2-proof}
    v^{\WA} (K) = \sup_{\Pw \in \D^{\WA}(\epsilon,\Sigma_0)}  \EE_{\Pw}[g_K(w)].
\end{align}
In what follows, we are going to expand the ambiguity set (by removing the restrictions on $\Pw$ to be zero-mean) and consider the problem 
\begin{align}\label{eq:v_W2-proof-relaxed}
    \ol{v}^{\WA} (K) \Let \sup_{\Pw}  \big\{\EE_{\Pw}[g_K(w)] \;:\; \Pw \in \mathcal{B}(\R^p),\; \text{dist}^{\WA}(\Pw \| \Pw_0) \leq \epsilon  \big\}.
\end{align}
As we shall see, the optimal solution of the preceding ``relaxed'' problem~\eqref{eq:v_W2-proof-relaxed} is a zero-mean Gaussian distribution belonging to the set $\D^{\WA}(\epsilon,\Sigma_0) \subseteq \PS$. 
Therefore, the optimal solution to the original problem~\eqref{eq:v_W2-proof} is the same distribution. 

The preceding problem~\eqref{eq:v_W2-proof-relaxed} is also a well-studied case in DRO. 
In particular, via the Kantorovich duality approach, we have (see, e.g., \cite[Thm.~1 and Rem.~1]{blanchet2019quantifying})
\begin{align*}
    \ol{v}^{\WA} (K) = \inf_{\theta \ge 0}\Big\{\epsilon^2\theta + \EE_{\Pw_0}\Big[\sup_{z\in\R^p}\,\big\{g_K(z)-\theta\|z-w\|^2\big\}\Big]\Big\}.
\end{align*}
For fixed $w$, the objective of the inner maximization is finite precisely when
$\theta I-M_K\succ0$, i.e., $\theta > \lambda^{\WA}_K$. 
In that case, the objective is a strictly concave quadratic function in $z$, and the unique maximizer is $z_K(\theta,w)= \theta(\theta I-M_K)^{-1}w$. 
Substitution then gives the optimal value
\begin{align*}
    \sup_z\{z\tr M_K z - \theta\|z-w\|^2\}
    &= w\tr\left(\theta^2(\theta I-M_K)^{-1}-\theta I\right)w \\
    &= w\tr\theta M_K(\theta I-M_K)^{-1}w.
\end{align*}
Since $\Pw_0=\mathcal{N}(0,\Sigma_0)$, taking expectation yields 
$\ol{v}^{\WA} (K) = \min_{\theta} \{ \phi^{\WA}(K,\theta) \;:\; \theta > \lambda^{\WA}_K \}$. 
By Lemma~\ref{lem: phi W2}, the preceding problem admits a unique dual minimizer $\theta^{\WA}_K$. 
Moreover, the worst-case distribution can be characterized uniquely by the linear map 
$T_K(w) = z_K(\theta^{\WA}_K,w) = \theta^{\WA}_K(\theta^{\WA}_K I-M_K)^{-1}w$
\cite[Rems.~2 and 3]{blanchet2019quantifying}. 
That is, the worst-case distribution in the relaxed problem~\eqref{eq:v_W2-proof-relaxed} is the pushforward $\Pw^{\WA}_K=(T_K)_{\#}\Pw_0$, which is the zero-mean Gaussian with covariance $\Sigma_K^{\WA}$ as in \eqref{eq:DR-W2-covariance}. 
This, in turn, implies that $\Pw^{\WA}_K$ is also the optimal solution of the original problem~\eqref{eq:v_KL-proof}.  
Finally, using the primal moment formulation, we have 
$v^{\WA}(K) = v(K,\Pw^{\WA}_K) =  \langle \Sigma^{\WA}_{K} , M_K\rangle$. 

We next show that $v^{\WA}$ is real-analytic on $\Ks$. 
The main ingredient is again showing that $K\mapsto\theta^{\WA}_K$ is real-analytic. 
Consider
\begin{align*}
    \Phi^{\WA}(K,\theta) \Let \frac{\partial}{\partial\theta}\phi^{\WA}(K,\theta) = \epsilon^2 - \Trc\!\left( M_K(\theta I-M_K)^{-1} \Sigma_0 M_K(\theta I-M_K)^{-1} \right),
\end{align*}
on the open set $(K,\theta)\in\Ks\times(\lambda^{\WA}_K,\infty)$. 
Recall that $M_K$ is a real-analytic function on $K\in\Ks$; see the proof of Lemma~\ref{lem:standard_H2_grad}. 
Then, since $\theta I-M_K \succ 0$ for $\theta > \lambda^{\WA}_K = \lambda_{\max}(M_K)$ and matrix inversion is real-analytic on the space of positive definite matrices, $\Phi^{\WA}$ is real-analytic. 
By Lemma~\ref{lem: phi W2}, $\Phi^{\WA}(K,\theta^{\WA}_K)=0$ and $\frac{\partial}{\partial\theta}\Phi^{\WA}(K,\theta^{\WA}_K)>0$. 
The real-analytic implicit function theorem (see, e.g., \cite[Thm.~2.4.4]{krantz2002implicit}) therefore gives a local
real-analytic representation of $\theta^{\WA}_K$ around every $K\in\Ks$.  Since the
scalar minimizer is unique, these local representations agree on overlaps;
hence $K\mapsto\theta^{\WA}_K$ is real-analytic on $\Ks$. 
It follows from~\eqref{eq:DR-W2-covariance} that $K\mapsto \Sigma_K^{\WA}$ is
real-analytic. 
Thus, $v^{\WA}(K)=\langle \Sigma_K^{\WA},M_K\rangle$ is real-analytic on $\Ks$.
\end{proof}

\subsection{DR optimality of standard $\mathcal{H}_2$-optimal gain}\label{sec:ambiguity_compare}

The three ambiguity sets satisfy Assumption~\ref{as:ambiguity}.
They contain the nominal distribution and are independent of $K$.
For the Frobenius set,
$\|\Sp\|_F\leq\|\Sigma_0\|_F+\epsilon$.
For the KL set, we can use the relative-entropy variational
inequality in \cite[Sec.~2.1]{hu2013kullback} with
$f(w)=\|w\|_2^2/\theta$ and any fixed
$\theta>2\lambda_{\max}(\Sigma_0)$ to obtain
\[
 \sup_{\Pw\in\D^{\KL}(\epsilon,\Sigma_0)}
 \EE_{\Pw}[\|w\|_2^2]
 \leq \epsilon\theta-\frac{\theta}{2}
 \log\det\big(I-\frac{2}{\theta}\Sigma_0\big)
 <\infty.
\]
For the W2 set, the triangle inequality \cite[Ch.~6]{villani2009optimal}, with the Dirac
measure $\mathsf{D}_0$ concentrated entirely at the origin, gives
\[
 \bigl(\EE_{\Pw}[\|w\|_2^2]\bigr)^{1/2}
 =\text{dist}^{\WA}(\Pw,\mathsf{D}_0)
 \leq \text{dist}^{\WA}(\Pw,\Pw_0) + \text{dist}^{\WA}(\Pw_0,\mathsf{D}_0)
 =\epsilon+\sqrt{\Trc(\Sigma_0)}.
\]
Since the distributions are centered,
$\|\Sp\|_F\leq\Trc(\Sp)=\EE_{\Pw}[\|w\|_2^2]$,
which establishes the required uniform covariance bounds.
Therefore, the results of Section~\ref{sec:DR-H2} also hold for the DR $\mathcal{H}_2$ problems~\eqref{eq:dr_frob}, \eqref{eq:dr_KL}, and \eqref{eq:dr_W2}. 
To be precise, we have the following results for $\mathrm{a}\in\{\mathrm{F},\, \mathrm{KL},\, \mathrm{W}\}$:
\begin{itemize}
\item By Proposition~\ref{prop:DR_H2_Clarke}, the gradient of the corresponding DR cost is given by
\begin{equation}\label{eq:dr_gradients}
    \nabla_K v^{\mathrm{a}}(K) = 2 S_K \Wc^{\mathrm{a}}_K, \quad \forall K\in\Ks,
\end{equation}
where $\Wc^{\mathrm{a}}_K \Let \Wc_K(\Sigma^{\mathrm{a}}_K)$ is the solution to equation~\eqref{eq:Lyap_eq_control} with $\Sigma = \Sigma^{\mathrm{a}}_K$.

\item By Theorem~\ref{thm:DR_H2_optimal_grad}, a stabilizing gain $K\opt\in\Ks$ is a solution of the corresponding DR $\mathcal{H}_2$ problem if and only if $\nabla_K v^{\mathrm{a}}(K\opt) = 0$. 

\item By Theorem~\ref{thm:dr_H2_standard}, if $K\opt\in\Ks$ is a standard $\mathcal{H}_2$-optimal gain, then $K\opt$ is also a solution of the DR $\mathcal{H}_2$ problems with the corresponding worst-case covariance $\Sigma^{\mathrm{a}}_{K\opt}$ and optimal value
\begin{equation}\label{eq:dr_value}
    v^{\mathrm{a}}\opt = \langle\Sigma^{\mathrm{a}}_{K\opt},M_{K\opt}\rangle.
\end{equation} 
    
\end{itemize}

The last item above once again indicates that the standard $\mathcal{H}_2$-optimal solution, if it exists, also solves the DR $\mathcal{H}_2$ problems considered in this subsection. 
The difference between the three cases is however in the corresponding worst-case covariance and optimal DR cost. 
See Appendix~\ref{app:compare-ambiguity} for a discussion on the differences between the three choices of the uncertainty models.

\subsection{Numerical illustration of performance certification}
\label{sec:numerical}

In this subsection, we use a synthetic numerical example to illustrate the worst-case performance certificates of Section~4 for a controller obtained from standard $\mathcal H_2$ synthesis. 
We consider a system with $n=10$ states, $m=3$ inputs, and $p=5$ disturbance channels. 
The state matrix~$A$ is randomly generated and normalized to be unstable with the spectral radius $\rho(A) = 2$;
the input matrix~$\Bu$ is randomly generated such that the pair~$(A, \Bu)$ is controllable;
the disturbance matrix is $\Bw$ randomly generated;
the performance matrices $C$ and $D$ are set such that $z_t = [x_t\tr \ \ u_t\tr]\tr $. 
A solution to the standard $\mathcal{H}_2$ synthesis problem~\eqref{eq:standard_H2_problem}
is then the LQ-optimal gain $K_{\mathrm{LQ}}$ for the deterministic dynamics $x_{t+1}=Ax_t+Bu_t$ and state- and input-cost weighting matrices $Q = I$ and $R =I$ respectively; see Remark~\ref{Rem:standrad_H2_optimal_suff}. 
We note that $\rho(A_{K_{\mathrm{LQ}}}) = 0.53$.

Using the results of Section~\ref{sec:ambiguity_compare}, we can now verify if this standard $\mathcal{H}_2$-optimal gain solves the DR $\mathcal{H}_2$ problems corresponding the Frobenius, KL, and W2 ambiguity sets. 
To this end, we consider the DR $\mathcal{H}_2$ problems~\eqref{eq:dr_frob}, \eqref{eq:dr_KL}, and, \eqref{eq:dr_W2}, with radius $\epsilon \in (0 , 1]$ and nominal covariance $\Sigma_0 = (LL\tr)/\rho(LL\tr)$, where $L\in\R^{p\times p}$ is a randomly generated full-rank matrix. 
Computing the gradient of the corresponding standard and DR costs at $K_{\mathrm{LQ}}$ using~\eqref{eq:dr_gradients} resulted in $\|\nabla f(K_{\mathrm{LQ}})\|_F \leq 10^{-10}$ for all $f\in\{v^{\St},\,v^{\F},\,v^{\KL},\,v^{\WA}\}$. 
That is, the gain $K_{\mathrm{LQ}}$ is also a stationary point of the DR costs, and hence, by Theorem~\ref{thm:DR_H2_optimal_grad}, solves the corresponding DR $\mathcal{H}_2$ problems. 
We note that for the KL and W2 ambiguity sets, the optimal dual variables $\theta^{\KL}_{K\opt}$ and $\theta^{\WA}_{K\opt}$ in Lemmas~\ref{lem: phi KL} and ~\ref{lem: phi W2}, respectively, are computed using bisection. 
Figure~\ref{fig:epsilon_compare} depicts the optimal value of these problems, computed using~\eqref{eq:dr_value}, for different values of the radius $\epsilon \in (0 , 1]$. 
As expected, the Frobenius DR cost shows a linear sensitivity. The KL DR cost has higher sensitivity for smaller values of $\epsilon$, while that is the case for the W2 DR cost for larger values of $\epsilon$.

\begin{figure}[t]
\centering
\includegraphics[width=0.35\textwidth]{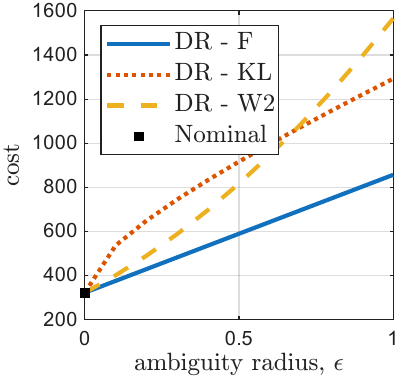}
\caption{Sensitivity of the DR costs to ambiguity radius.}
\label{fig:epsilon_compare}
\end{figure}

\section{Final Remarks}
\label{sec:conclusion}

In this paper, we considered the distributionally robust extension of the $\mathcal{H}_2$ synthesis problem for discrete-time LTI systems. 
We first established global landscape properties for the standard problem and then used the covariance representation of the generalized cost to analyze ambiguity sets with zero mean and uniformly bounded covariances.
The resulting DR objective may be nonsmooth, however it is locally Lipschitz and Clarke regular, and Clarke stationarity is both necessary and sufficient for global optimality.
We also showed that any attained standard $\mathcal{H}_2$ optimum is simultaneously optimal for every such DR problem. 
Moreover, we derived explicit worst-case models for spectral, Frobenius, KL, and W2 ambiguity sets. 
We conclude with several remarks on limitations of these results.

\textbf{Output-feedback synthesis.}
The analysis relies essentially on static state feedback, $u_t=Kx_t$.
For static output feedback, $u_t=Ky_t$, the closed-loop matrix has the form $A+\Bu K C_{\mathrm{y}}$, and the corresponding stabilizing set is generally nonconvex and may be \emph{disconnected}~\cite{feng2020connectivity}.
Moreover, output-feedback policy optimization can exhibit spurious stationary points, including nonglobal local minima and saddle points~\cite{duan2023optimization}.
The trajectory-orthogonality and matrix-minimality arguments used here therefore do not transfer directly.
Identifying additional structural conditions under which a componentwise or global benign-landscape result survives for static or dynamic output feedback remains an important open problem.

\textbf{Other uncertainty models.} 
Our distributional model assumes i.i.d.\ zero-mean disturbances with finite covariance.
Under these assumptions, steady-state quadratic performance depends on the disturbance distribution only through a single covariance matrix, which is precisely what enables the universal optimality result.
The Frobenius, KL, and Wasserstein models considered here represent substantially different geometries.
Nonetheless, all preserve this covariance reduction.
Natural extensions include nonzero or uncertain means, temporally correlated and nonstationary disturbances, ambiguity sets for entire disturbance processes, and simultaneous uncertainty in the system matrices.
In these settings the cost need not reduce to $\langle\Sigma,M_K\rangle$, and a standard $\mathcal{H}_2$ optimizer need not remain distributionally robust.
Another important direction is the statistical calibration of the ambiguity radius from finite trajectory data, which would connect the synthesis guarantees developed here to out-of-sample performance guarantees.

\textbf{Restrictions on the controller.}
The common-optimizer result holds over the full stabilizing state-feedback set $\Ks$.
Additional restrictions, such as sparsity or gain bounds, may exclude all unrestricted standard $\mathcal H_2$ optimizers. 
In this case, minimizing the nominal scalar cost over the restricted set need not minimize $M_K$ in the Loewner order. 
Feasible gains may therefore trade performance across disturbance directions, and the worst-case covariance can change their ranking. 
Thus, the restricted nominal and DR problems need not share an optimizer, although a difference is not guaranteed. 
Conversely, any unrestricted standard optimizer that remains feasible retains its matrix-minimality property and hence remains optimal for both restricted problems.

\appendix

\section{Auxiliary results}

\subsection{Invisible directions in standard $\mathcal{H}_2$ cost}\label{app:invisible}

The following result discusses the invisible directions of the standard $\mathcal{H}_2$ cost $v^{\St}$ using the time-domain characterization~\eqref{eq:standard-H2-cost-markov} via closed-loop Markov parameters. 
Recall that $\Wc_K^{\St}$ is the solution to the Lyapunov equation~\eqref{eq:Lyap_eq_control} associated with unit disturbance covariance, that is, the controllability Gramian of the pair~$(\Acl_K,\Bw)$. 

\begin{Prop}[Invisible directions in standard $\mathcal H_2$ cost]
\label{prop:invisible_direction}
Let $K\in\Ks$ and define
\begin{equation*}
    \mathcal{I}_K \Let \left\{ \Delta\in\R^{m\times n}: \begin{bmatrix} \Bu \\ D \end{bmatrix} \Delta \Wc_K^{\St}=0 \right\}.
\end{equation*}
Then, for every $\Delta\in\mathcal{I}_K$ such that $K+\Delta\in\Ks$, we have $v^{\St}(K+\Delta)=v^{\St}(K)$.
\end{Prop}

\begin{proof}
Since $\Acl_K$ is Schur stable, the range of its controllability
Gramian is equal to the reachable subspace of the pair~$(\Acl_K,\Bw)$, that is,
$\range(\Wc_K^{\St}) = \text{span} \begin{bmatrix} \Bw & \Acl_K\Bw & \Acl_K^2\Bw & \cdots & \Acl_K^{n-1}\Bw \end{bmatrix}$. 
In particular, 
$\range(\Acl_K^t\Bw) \subseteq \range(\Wc_K^{\St})$ for all $t\in\N_0$. 
Let $\Delta\in\mathcal{I}_K$ so that $\Bu\Delta\Wc_K^{\St}=0$ and $D\Delta\Wc_K^{\St}=0$. 
Then, using the preceding inclusion, we arrive at
\begin{equation}\label{eq:invisible_on_reachable_subspace}
    \Bu\Delta\Acl_K^t\Bw=0 \quad \text{and} \quad  D\Delta\Acl_K^t\Bw=0, \quad \forall t\in\N_0.
\end{equation}
We now prove by induction that
\begin{equation}\label{eq:equal_state_markov_parameters}
    \Acl_{K+\Delta}^t\Bw = \Acl_K^t\Bw, \quad \forall t\in\N_0.    
\end{equation}
The equality is immediate for $t=0$. 
Suppose that it holds for some $t\in\N_0$. 
Then, using the first equality in \eqref{eq:invisible_on_reachable_subspace}, we obtain
\begin{align*}
    \Acl_{K+\Delta}^{t+1}\Bw = \Acl_{K+\Delta} \Acl_{K+\Delta}^t\Bw
    = (\Acl_K+\Bu\Delta)\Acl_K^t\Bw
    = \Acl_K^{t+1}\Bw + \Bu\Delta\Acl_K^t\Bw 
    = \Acl_K^{t+1}\Bw.
\end{align*}
Finally, using~\eqref{eq:equal_state_markov_parameters} and the second equality in~\eqref{eq:invisible_on_reachable_subspace}, we have
\begin{align*}
    \Ccl_{K+\Delta}\Acl_{K+\Delta}^t\Bw = (\Ccl_K+D\Delta)\Acl_K^t\Bw \nonumber
    = \Ccl_K\Acl_K^t\Bw +D \Delta\Acl_K^t\Bw = \Ccl_K\Acl_K^t\Bw, \quad \forall t\in\N_0.
\end{align*}
That is, all closed-loop Markov parameters from the disturbance $w$
to the performance output $z$ remain unchanged. 
The claim then follows from the time-domain characterization~\eqref{eq:standard-H2-cost-markov} of $v^{\St}$ using the closed-loop Markov parameters.
\end{proof}

In simple words, the preceding result states that any stable perturbation~$\Delta$ of a stabilizing control gain~$K$ that is in the null space of $[\Bu\tr \; D\tr]\tr$ and/or the unreachable subspace of the pair $(\Acl_K,\Bw)$ will be invisible in the sense that it does not affect the value of the cost $v^{\St}(K)$. 
This is indeed expected since the closed-loop matrices $\Acl_K = A+\Bu K$ $\Ccl_K = C+D K$ are by construction insensitive to perturbations of $K$ in the null space of $[\Bu \;\; D]$. 
More importantly, $\range (\Wc_K^{\St})$ defines the reachable subspace of the pair $(\Acl_K,\Bw)$, that is, the part of the state space $\mathbb{R}^n$ that the disturbance $w$ is physically capable of exciting.
By the Kalman controllability decomposition, any state in the orthogonal complement $\range (\Wc_K^{\St})^{\perp}$ is decoupled from the disturbance. 
Hence, perturbations of $K$ in $\range (\Wc_K^{\St})^{\perp}$ act only on the unreachable modes of the pair $(\Acl_K,\Bw)$ and leave the input-output behavior of the disturbance channel invariant. 

\subsection{From DR to standard $\mathcal{H}_2$ optimality}\label{app:DR_stand_optimality}

The following result provides a sufficient condition for a solution of the DR $\mathcal{H}_2$ problem problem to be also optimal for the standard $\mathcal{H}_2$ problem. 
Recall that $\mathcal{C}_{K} = \argmax_{\Sigma\in\mathcal{C}}\; \langle \Sigma,M_K\rangle$ denotes 
the set of active worst-case covariances corresponding to $K\in\Ks$ in the DR $\mathcal{H}_2$ problem. 
Also, recall that $\Wc_K (\Sigma)$ is the solution to the Lyapunov equation~\eqref{eq:Lyap_eq_control} associated with disturbance covariance $\Sigma$.

\begin{Prop}[From DR to standard $\mathcal H_2$ optimality] \label{prop:DR_optimal_implies_standard}
Let Assumption~\ref{as:ambiguity} hold and $K\opt\in\Ks$ be a stabilizing solution of the DR $\mathcal{H}_2$ problem~\eqref{eq:dr_H2_problem}. 
Also, let $\Sigma_{K\opt} \in \mathcal{C}_{K\opt}$ be an active worst-case covariance corresponding to $K\opt$ such that $S_{K\opt}\Wc_{K\opt} (\Sigma_{K\opt})=0$. 
If $\Sigma_{K\opt}\succ0$, then the gain $K\opt$ is also a solution of the standard $\mathcal{H}_2$ problem~\eqref{eq:standard_H2_problem}.
\end{Prop}

\begin{proof}
By Theorem~\ref{thm:DR_H2_optimal_grad}, since $\Ks$ is a DR $\mathcal{H}_2$-optimal solution, we have $0\in\partial_{\rm C}v^{\DR}(K\opt)$. 
Then, by the Clarke subdifferential representation in Proposition~\ref{prop:DR_H2_Clarke}, there exists an active worst-case covariance $\Sigma_{K\opt} \in \mathcal{C}_{K\opt}$ such that $S_{K\opt}\Wc_{K\opt} (\Sigma_{K\opt})=0$. 
Post-multiplying the latter equality by $S_{K\opt}\tr$ and taking the trace, we obtain
$\Trc \big(S_{K\opt}\, \Wc_{K\opt} (\Sigma_{K\opt})\, S_{K\opt}\tr \big)=0$. 
Since $K\opt\in\Ks$, the matrix $\Acl_{K\opt}$ is Schur stable and the controllability Gramian $\Wc_{K\opt} (\Sigma_{K\opt})$ admits the convergent series representation~\eqref{eq:contr_gram}. 
Plugging in the series representation of $\Wc_{K\opt} (\Sigma\opt)$ in the equality above, we arrive at
\begin{align*}
    \sum_{t=0}^{\infty} \Trc\!\left(S_{K\opt}\Acl_{K\opt}^t\Bw\, \Sigma_{K\opt}\Bw\tr(\Acl_{K\opt}^t)\tr S_{K\opt}\tr \right) 
    = \sum_{t=0}^{\infty} \left\| S_{K\opt}\Acl_{K\opt}^t\Bw\,\Sigma_{K\opt}^{1/2} \right\|_F^2 = 0.
\end{align*}
Every summand in the preceding equality is nonnegative. Hence, 
$S_{K\opt}\Acl_{K\opt}^t\Bw\,\Sigma_{K\opt}^{1/2} = 0$ for all $t\in\N_0$. 
By assumption $\Sigma_{K\opt}$ and hence $\Sigma_{K\opt}^{1/2}$ is nonsingular.
Then, pre-multiplying the preceding equality by $\Sigma_{K\opt}^{-1/2}$ gives
\begin{equation}\label{eq:unweighted_reachable_stationarity}
     S_{K\opt}\Acl_{K\opt}^t\Bw = 0, \quad \forall t\in\N_0.
\end{equation}
On the other hand, using the gradient formula for the standard $\mathcal H_2$ cost in Lemma~\ref{lem:standard_H2_grad} and the convergent series representation~\eqref{eq:contr_gram} of standard controllability Gramian $\Wc_{K\opt}^{\St} = \Wc_{K\opt}(I)$, we have
\begin{equation*}
    \nabla_Kv^{\St}(K\opt) = 2\,S_{K\opt}\Wc_{K\opt}^{\St} = 2 \sum_{t=0}^{\infty} \big(S_{K\opt}\Acl_{K\opt}^t\Bw\big) \Bw\tr(\Acl_{K\opt}^t)\tr.
\end{equation*}
Using the equality \eqref{eq:unweighted_reachable_stationarity}, we conclude that $\nabla_K v^{\St}(K\opt) = 0$. 
Then, by Theorem~\ref{thm:standard_H2_optimal_grad}, $K\opt$ is a standard $\mathcal{H}_2$-optimal solution. 
\end{proof}

We notice that the nonsingularity of $\Sigma_{K\opt}$ is used precisely in passing from $S_{K\opt}\Acl_{K\opt}^t\Bw\,\Sigma_{K\opt}^{1/2}=0$ to $S_{K\opt}\Acl_{K\opt}^t\Bw=0$ for all $t\in\N_0$.
If $\Sigma_{K\opt}$ is singular, the DR $\mathcal H_2$-optimality condition only constrains the disturbance directions contained in $\range(\Sigma_{K\opt})$. 
Therefore, the directions contained in the null space of $\Sigma_{K\opt}$ remain invisible, and the standard $\mathcal H_2$ optimality does not follow.

\subsection{Comparison of the three ambiguity sets}\label{app:compare-ambiguity}

In what follows, we provide a brief discussion on the differences between the three choices of the uncertainty models discussed in Section~\ref{sec:specific-ambiguity-set}. 

Recall that the three worst-case covariance matrices admit the following
representations:
\begin{subequations}
\begin{align}
    \Sigma_K^{\F}  &= \Sigma_0 + \frac{\epsilon}{\|M_K\|_{F}}M_K, \label{eq:wc_cov_frob_comparison}\\
    \Sigma_K^{\KL} &= \Sigma_0^{1/2} \big( I-2 (\theta_K^{\KL})^{-1} R_K \big)^{-1} \Sigma_0^{1/2}, \quad R_K \Let \Sigma_0^{1/2}M_K\Sigma_0^{1/2}, \label{eq:wc_cov_KL_comparison}\\
    \Sigma_K^{\WA} &= T_K\,\Sigma_0\,T_K\tr, \quad T_K \Let \big(I-(\theta_K^{\WA})^{-1} M_K\big)^{-1},
    \label{eq:wc_cov_W2_comparison}
\end{align}
\end{subequations}
Let us first notice that in all three cases, the DR constraint is active and the worst-case distribution lies at the boundary of ambiguity set (under Assumption~\ref{as:non-zero-M_K}). 
In particular, we have
\begin{itemize}
    \item $\left\|\Sigma^{\F}_K- \Sigma_0 \right\|_{F} = \epsilon $ for the Frobenius ambiguity set in Proposition~\ref{prop: DR cost frob},
    \item $\text{dist}^{\KL}(\Pw^{\KL}_K,\Pw_0) = \epsilon$ for the KL ambiguity set in Proposition~\ref{prop: DR cost KL}, and, 
    \item $\text{dist}^{\WA}(\Pw^{\WA}_K,\Pw_0) = \epsilon$ for the W2 ambiguity set in Proposition~\ref{prop: DR cost W2}.
\end{itemize}

The Frobenius adversary performs an additive displacement of the nominal covariance $\Sigma_0$ in the direction $M_K$. 
Indeed, $M_K/\|M_K\|_{F}$ is the steepest-ascent direction of the linear functional $\Sigma\mapsto\langle \Sigma, M_K\rangle$ with respect to the Frobenius geometry. 
Notice that the direction of this displacement does not depend on $\Sigma_0$. 
Since the cost depends only on the covariance, every zero-mean distribution with covariance
$\Sigma_K^{\F}$ is worst case, e.g., the Gaussian representative $ \Pw_K^{\F} = \mathcal{N}(0,\Sigma_K^{\F})$.

For the KL ambiguity set, the worst-case distribution is uniquely obtained by exponentially tilting the nominal Gaussian distribution. 
The worst-case distribution remains Gaussian and satisfies $\Pw_K^{\KL} = \mathcal{N}(0,\Sigma_K^{\KL})$. 
Observe that the covariance increment in this case can be written as
\begin{equation*}
    \Sigma_K^{\KL}-\Sigma_0 = \Sigma_0^{1/2} \left[\left( I-2 (\theta_K^{\KL})^{-1} R_K \right)^{-1} -I \right] \Sigma_0^{1/2} \succeq 0.
\end{equation*}
Thus, the KL adversary produces a nonlinear multiplicative inflation in whitened coordinates. 
The relevant directions are the eigendirections of $R_K$, which jointly account for the performance sensitivity $M_K$ and the nominal covariance $\Sigma_0$. 

For the W2 ambiguity set, the worst-case distribution is the pushforward of the nominal distribution through the linear transport map $T_K$, that is, $\Pw_K^{\WA} = (T_K)_{\#}\Pw_0 = \mathcal{N}(0,\Sigma^{\WA}_K)$. 
The Wasserstein perturbation is therefore fundamentally different from the Frobenius displacement and the KL exponential tilt: 
It acts on disturbance realizations before forming their covariance. 
In particular, the difference $\Sigma_K^{\WA}-\Sigma_0$ need not be positive definite when $M_K$ and $\Sigma_0$ do not commute. 
Thus, Wasserstein transport may both rotate the nominal covariance ellipsoid and redistribute variance
among its directions, while increasing the value $\langle \Sigma, M_K\rangle$.

The distinction between the three ambiguity sets becomes particularly transparent when $M_K$ and
$\Sigma_0$ commute. Suppose that $M_K$ and
$\Sigma_0$ are simultaneously diagonalized
as
\[
    M_K=U\,\text{diag}(\lambda_1,\ldots,\lambda_p)\,U\tr,
    \quad
    \Sigma_0=U\,\text{diag}(\sigma_1,\ldots,\sigma_p)\,U\tr.
\]
Then, all three worst-case covariance matrices have the same eigenvectors with eigenvalues 
\begin{align*}
    \sigma_i^{\F} = \sigma_i+ \frac{\epsilon\lambda_i} {\big(\sum_{j=1}^p \lambda_j^2\big)^{1/2}}, \quad
    \sigma_i^{\KL} = \frac{\sigma_i} {1-2\sigma_i \lambda_i/\theta_K^{\KL}}, \quad 
    \sigma_i^{\WA} = \frac{\sigma_i} {(1-\lambda_i/\theta_K^{\WA})^2}.
\end{align*}
Hence, the Frobenius model produces an additive covariance
perturbation, the KL model produces a multiplicative inflation
determined jointly by $\sigma_i$ and $\lambda_i$, and the Wasserstein model
stretches the disturbance realization in direction $i$ by the factor
$(1-\lambda_i/\theta_K^{\WA})^{-1}$.

\bibliographystyle{apalike} 
\begin{small}
\bibliography{ref}
\end{small}

\end{document}